\documentclass[12pt]{amsart}
\usepackage{amscd,amssymb,epsf}
\usepackage{amsfonts,latexsym}
\usepackage{epsf,graphicx}

\newtheorem{theorem}{Theorem}[section]
\newtheorem{cor}{Corollary}[section]
\newtheorem{defin}{Definition}[section]
\newtheorem{lemma}{Lemma}[section]

\newcommand{\dbar}{\overline{\partial}}
\newcommand{\dz}{\partial_z}
\newcommand{\dzb}{\overline{\partial}_{z}}

\newcommand{\dkb}{\overline{\partial}_k}

\newcommand{\ra}{\rightarrow}

\newcommand{\kp}{k^{\prime}}

\newcommand{\SI}{{\mathcal{S}}}

\newcommand{\TT}{{\mathcal{T}}}
\newcommand{\QQ}{{\mathcal{Q}}}
\newcommand{\QQp}{{\mathcal{Q}}^+}
\newcommand{\QQm}{{\mathcal{Q}}^-}

\newcommand{\T}{{\mathbf{t}}}

\newcommand{\fii}{\varphi}
\newcommand{\eps}{\varepsilon}
\newcommand{\R}{{\mathbb R}}

\newcommand{\C}{{\mathbb C}}
\newcommand{\CC}{{\mathcal{C}}}

\newcommand{\qis}{q}
\newcommand{\qnv}{q^{\mbox{\tiny \rm NV}}}

\numberwithin{equation}{section}

\begin{document}

\title[Novikov-Veselov equation and the ISM]{ The Novikov-Veselov Equation and the \\  Inverse Scattering Method, Part I: Analysis}

\author{M Lassas, J L Mueller, S Siltanen and A Stahel}

\begin{abstract}
The Novikov-Veselov (NV) equation is a
  (2+1)-dimensional nonlinear evolution equation that generalizes the
  (1+1)-dimensional Korteweg-deVries (KdV) equation. Solution of the NV equation using the inverse scattering method has been discussed in the  literature, but only formally (or with smallness assumptions in case of nonzero energy) because of the possibility of exceptional points, or singularities in the scattering data. In this work, absence of exceptional points is proved at zero energy for  evolutions with compactly supported, smooth and rotationally symmetric initial data of the conductivity type: $q_0=\gamma^{-1/2}\Delta\gamma^{1/2}$ with a strictly positive function $\gamma$. The inverse scattering evolution is shown to be well-defined, real-valued, and preserving conductivity-type. There is no  smallness assumption on the initial data. 
\end{abstract}

\maketitle

Version 4 (submitted), May 19, 2011

\setcounter{tocdepth}{1}


\section{Introduction}

\noindent
The nonlinear Novikov-Veselov equation for $q_\tau=q_\tau(z)=q_\tau(x,y)$ is 
\begin{equation}  \label{evolution}
  \left\{
  \begin{array}{rcl}
  \displaystyle 
  \frac{\partial q_\tau}{\partial \tau} &=& 
  \displaystyle 
  -\dz^3q_\tau -\dzb^3q_\tau + \frac{3}{4}\,\dz(q_\tau v_\tau) + \frac{3}{4}\,\dzb(q_\tau\overline{v}_\tau),\\
  \\
   v_\tau(z)   &=&  \dzb^{-1}\dz q_\tau(z),
\end{array}
\right.
\end{equation}
where $\tau\geq 0$ and $\dzb=\frac 12(\frac{\partial}{\partial x} + i\frac{\partial}{\partial y})$. Equation (\ref{evolution})  is a (2+1)-dimensional generalization of the celebrated Korteweg-de Vries (KdV) equation \cite{KdV}. In 1984, S.~P.~Novikov and Veselov introduced (\ref{evolution}) in a periodic setting in \cite{novikovveselov,veselovnovikov}  as a continuation of the work by Nizhnik \cite{nizhnik}.  Equation (\ref{evolution}) is the most natural
generalization of the KdV equation to dimension (2+1) since the variables $x$ and $y$ have more symmetric roles in (\ref{evolution})
than they do in other generalizations, such as the Kadomtsev-Petviashvili equation, see \cite{bogdanov}.

The study of equation (\ref{evolution}) in the non-periodic setting ($z\in\R^2$) via the inverse scattering method was initiated by Boiti, Leon, Manna
and Pempinelli \cite{boiti86,boiti} and continued by Tsai \cite{tsai,tsai2,tsai3}. They discuss the following formal inverse scattering scheme for solving the Cauchy problem for (\ref{evolution}):
\begin{equation}\label{diagram}
\begin{picture}(200,122)
\thinlines
 \put(30,105){\vector(1,0){130}}
 \put(0,102){$\T^+_0(k)$}
 \put(165,102){$\T^+_\tau(k)$}
 \put(55,110){$\exp(i \tau(k^3+\overline{k}^3))\cdot$}
 \put(7,15){\vector(0,1){77}}
 \put(13,92){\vector(0,-1){77}}
 \put(17,67){$\QQp$}
 \put(-12,67){$\TT^+$}
 \put(0,0){$q_0(z)$}
 \put(156,67){$\TT^+$}
 \put(175,46){\vector(0,1){47}}
 \put(181,92){\vector(0,-1){47}}
 \put(184,67){$\QQp$}
 \put(165,28){$\qis_{\tau}(z)$}
 \put(30,3){\vector(1,0){130}}
 \put(40,7){{\footnotesize nonlinear evolution (\ref{evolution})}}
 \put(165,0){$\qnv_{\tau}(z)$,}
\end{picture}
\end{equation}
where $\TT^+$ and $\QQp$ stand for the direct and inverse nonlinear Fourier transform, respectively, and the function $\T^+_\tau:\C\ra \C$ is called the {\em scattering transform}. Precise definitions of  $\TT^+$, $\QQp$ and  $\T^+_\tau$ are given below in Section \ref{sec:ISM}. 
The definition of $\qis_{\tau}$ involves point-wise multiplication in the transform domain:
\begin{equation}\label{qisdef}
  \qis_{\tau}:=\QQp\Big(e^{i \tau(k^3+\overline{k}^3)}\,\T^+_0(k)\Big),
\end{equation}
and $\qnv_{\tau}$ is defined as the solution of (\ref{evolution}) with initial condition $\qnv_0=q_0$. 

The diagram (\ref{diagram}) is written with the hope that for certain initial data $q_0$ all the maps in (\ref{diagram}) would be well-defined and  $\qis_{\tau}=\qnv_\tau$. Then the nonlinear Novikov-Veselov equation (\ref{evolution}) could be solved by a linear operation on the transform side, analogously to the celebrated inverse scattering method for the KdV equation \cite{GGKM}. However, so far the analysis of diagram (\ref{diagram}) is only formal, and the identity $\qis_{\tau}=\qnv_\tau$ has not been rigorously proved for any class of initial data.

Why is the inverse scattering method  (\ref{diagram}) so difficult to analyze? The key obstacle in the investigation of (\ref{diagram}) is the possibility of {\em exceptional points} of $\qis_{\tau}$. Exceptional points are values of the generalized frequency-domain variable $k\in\C$ at which the solutions to the related (non-physical) scattering problem are not unique, meaning that nonzero radiating solutions exist for zero incident field. At exceptional points the scattering data $\T^+_\tau(k)$ is not well-defined and possibly singular. The operator $\QQp$ is not defined for singular argument functions, which prevents the use of (\ref{diagram}) if there are exceptional points. Furthermore, taking a small initial potential $q_0$ does not save the day because the related Faddeev Green's function has a $\log |k|$ singularity at $k=0$. Consequently, Neumann series techniques cannot be used in general to prove the absence of exceptional points; currently one has to resort to (rather instable) Fredholm arguments.

The following class of {\em conductivity-type potentials} is useful as a source of initial data because such potentials do not have exceptional points \cite[Lemma 1.5]{nachman}.

\vspace{-.5mm}

\begin{defin}\label{def:condtype}
A potential $q\in L^p(\R^2)$ with $1<p<2$ is  of conductivity type if $q=\gamma^{-1/2}\Delta\gamma^{1/2}$ for some real-valued
$\gamma\in L^\infty(\R^2)$ satisfying $\gamma(z)\geq c > 0$ for almost every $z\in\R^2$ and $\nabla(\gamma^{1/2})\in L^p(\R^2)$.
\end{defin}

\vspace{-.5mm}

\noindent The term ``conductivity'' and the seemingly superficial square roots in definition \ref{def:condtype} come from the related studies of Calder\'on's inverse conductivity problem, see \cite{nachman}. We use these terms and notations here because it turns out that the inverse scattering evolution preserves conductivity-type, and the evolving conductivity may have a (yet unknown) physical interpretation.

Let us recall what is known about the diagram (\ref{diagram}).
\begin{itemize}
\item {\bf Boiti, Leon, Manna and Pempinelli 1987} \cite{boiti}: assume that $q_0$ is such that the solution $\qnv_{\tau}$ to (\ref{evolution}) exists and does not have exceptional points. Then the scattering data evolves as $\TT^+(\qnv_{\tau})=e^{i\tau(k^3+\overline{k}^3)}\TT^+(q_0)$. 
\item {\bf Tsai 1994} \cite{tsai3}: take $q_0$ from a certain class of small and rapidly decaying initial data (the class excludes conductivity-type potentials). Assume that $q_0$ has no exceptional points and that $\qis_{\tau}$ is well-defined by (\ref{qisdef}). Then  $\qis_{\tau}$ is a solution of the Novikov-Veselov equation (\ref{evolution}). 
\item {\bf Nachman 1996} \cite{nachman}: let $q_0$  be of conductivity type. Then $q_0$ does not have exceptional points and the scattering data $\TT^+(q_0)$ is well-defined.
\item {\bf L-M-S 2007} \cite{LMS}: let $q_0$ be a smooth, compactly supported conductivity-type potential with $\gamma\equiv 1$ outside supp$(q_0)$. Then $\QQp(\TT^+ q_0)=q_0$. Further, formula (\ref{qisdef}) gives a well-defined continuous function $\qis_{\tau}:\R^2\ra \C$ satisfying the estimate $|\qis_\tau(z)|\leq C(1+|z|)^{-2}$ for all $\tau> 0$. 
\end{itemize}
\noindent
Nachman's work paved the way for rigorous results: all studies about diagram (\ref{diagram}) published before \cite{nachman} were formal as they had to assume the absence of exceptional points without specifying acceptable initial data.

We remark that the above discussion concerns only the ``zero-energy case'' where initial data $q_0(z)$ tends to zero when $|z|\ra\infty$. There is a body of work concerning inverse scattering solutions for the Novikov-Veselov solution when the initial data tends to a nonzero constant at infinity, see \cite{grinevichmanakov,grinevichnovikov,grinevichnovikov2,grinevichnovikov3} and the review article \cite{grinevich2000}. However, those results are based on a smallness assumption of initial data being close enough to a nonzero constant function; then there are no exceptional points. This paper is concerned with the zero-energy case where smallness assumptions cannot be used for proving the absence of exceptional points.

We prove the following new results. Let $q_0$ be a real-valued, infinitely smooth, compactly supported con\-ducti\-vity-type potential with $\gamma\equiv 1$ outside supp$(q_0)$. If the initial data has the rotational symmetry $q_0(z)=q_0(|z|)$ for all $z\in\R^2$, then $\qis_\tau$ stays real-valued for all $\tau\geq 0$. Furthermore, $\qis_\tau$ is a conductivity-type potential and does not have exceptional points. Also, the scattering data of $\qis_\tau$ is well-defined and  the following identity holds:
\begin{equation}\label{scatidentity}
\TT^+(\QQp \T^+_\tau)=\T^+_\tau.
\end{equation}
The precise statement and assumptions are given below in Corollary \ref{conclcor}. We remark that there is no smallness assumption on the initial data. 


Our results are not completely restricted to compactly supported initial data satisfying $q_0(z)=q_0(|z|)$. For example, evolutions from symmetric initial data are valid starting points of another evolution; then the initial data is not necessarily symmetric or compactly supported. See Section \ref{sec:concl} for more details. 

In Part II of this paper we compare numerically $\qis_{\tau}$ and $\qnv_\tau$ for several examples with rotationally symmetric, compactly supported initial data of conductivity type. The evolutions $\qis_{\tau}$ and $\qnv_\tau$ are found to agree with high precision, suggesting that the equality $q_\tau=\qnv_{\tau}$ holds. However, rigorous proof of that identity remains an open theoretical problem; see Appendix \ref{appendix_Tsai}.

This paper is organized as follows. In section \ref{sec:ISM} we define the direct and inverse nonlinear Fourier transforms  $\TT^\pm$ and $\QQ^\pm$.  In section \ref{sec:CGOprops} we prove new estimates for the complex geometric optics solutions used in the inverse scattering method. Section \ref{sec:radial} contains results concerning rotationally symmetric initial data. Section \ref{sec:preserveconductivity} is devoted for a proof that if the evolved potential $\qis_\tau$ stays real-valued, then it stays conductivity-type. In Section \ref{sec:scatevol} we prove identity (\ref{scatidentity}), and in Section \ref{sec:concl} we conclude our results. 
Throughout the paper we denote $\langle z\rangle=1+|z|$ and abuse notation by writing $k_1+ik_2=k=(k_1,k_2)$ and $x+iy=z=(x,y)$.

\section{The inverse scattering method} \label{sec:ISM}

\subsection{The direct scattering maps $\TT^\pm$}

Consider the Schr\"odinger equation
\begin{equation}\label{Schrode}
  (-\Delta+q)\psi^{\pm}(\,\cdot\,,k)=0,
\end{equation}
where $q\in L^p(\R^2)$ is a real-valued potential with $1<p<2$ and $k$ is a complex parameter.  We look for {\em complex geometrical optics solutions} $\psi^{\pm}$ of (\ref{Schrode}) with asymptotic behavior 
\begin{eqnarray*}
 &&\psi^+(z,k)\sim e^{ikz}=\exp(i(k_1+ik_2)(x+iy)),\\
 &&\psi^-(z,k)\sim e^{ik\overline{z}}=\exp(i(k_1+ik_2)(x-iy)).
\end{eqnarray*} 
Such solutions were first introduced by Faddeev \cite{faddeev}. More precisely, we require
\begin{eqnarray}
  \label{exp_asymp}
  e^{-ikz}\psi^+(z,k)-1 &\in& W^{1,\tilde{p}}(\R^2), \\
  \label{exp_asympminus}
  e^{-ik\overline{z}}\psi^-(z,k)-1 &\in& W^{1,\tilde{p}}(\R^2),
\end{eqnarray}
where $1/\tilde{p}=1/p-1/2$.  Here $W^{1,\tilde{p}}(\R^2)$ is the Sobolev space consisting of $L^{\tilde{p}}(\R^2)$ functions whose (distributional) partial derivatives belong to $L^{\tilde{p}}(\R^2)$ as well. Note that $\tilde{p}>2$ and by Sobolev's imbedding theorem $W^{1,\tilde{p}}(\R^2)$ functions are continuous. 

Given a potential $q$, there may be complex numbers $k$ for which the solutions $\psi^{\pm}$ of (\ref{Schrode}) satisfying (\ref{exp_asymp}) and (\ref{exp_asympminus}) are not unique. Such $k$ are called {\em exceptional points} of $q$.

If $q$ has no exceptional points, then we write $\mu^+(z,k):=e^{-ikz}\psi^+(z,k)$ and $\mu^-(z,k):=e^{-ik\overline{z}}\psi^-(z,k)$ and write formally
\begin{eqnarray}\label{tdef}
  \T^+(k) &:=& \int_{\R^2} e^{i(kz+\overline{k}\overline{z})}q(z)\mu^+(z,k)dz, \\
  \label{tdefm}
    \T^-(k) &:=& \int_{\R^2} e^{i(\overline{k}z+k\overline{z})}q(z)\mu^-(z,k)dz.
\end{eqnarray}
If the integrals in (\ref{tdef}) and (\ref{tdefm}) are convergent, then $\T^\pm:\C\ra \C$ are well-defined and we set $\TT^\pm(q) = \T^\pm$. 

For example, compactly supported conductivity-type potentials do not have exceptional points and lead to convergent integrals in (\ref{tdef}) and (\ref{tdefm}), see \cite{nachman}.

Real-valuedness of $q$ results in symmetries in scattering data. Complex conjugating equation (\ref{Schrode}) for $\psi^+$ gives
\begin{equation}\label{schjorina1}
0=\overline{  (-\Delta+q)\psi^{+}(z,k) }=   (-\Delta+q)\overline{\psi^{+}(z,k)},
\end{equation}
and conjugating the corresponding asymptotic condition (\ref{exp_asymp}) gives
\begin{equation}\label{schjorina2}
\overline{ e^{-ikz}\psi^+(z,k)-1 } =  e^{-i(-\overline{k})\overline{z}}\overline{\psi^+(z,k)}-1 \in W^{1,\tilde{p}}(\R^2).
\end{equation}
Comparing (\ref{Schrode}) with (\ref{schjorina1}) and (\ref{exp_asympminus}) with (\ref{schjorina2}), and using uniqueness, shows that
\begin{equation} \label{pm_psi}
\overline{\psi^+(z,k)}=\psi^-(z,-\overline{k}).
\end{equation}
Furthermore, complex conjugating equation  (\ref{tdef}) and substituting (\ref{pm_psi})  yields
\begin{equation}\label{pm_tksym}
  \overline{\T^+(k)}=\T^-(-\overline{k}).
\end{equation}

\subsection{The inverse scattering maps $\QQ^\pm$}\label{sec:Qmaps}

Given two functions $\T^{\pm}:\C\ra\C$, consider the D-bar equations
\begin{eqnarray}
\label{dbar_eq_p}
  \frac{\partial}{\partial\overline{k}} \,\mu^+(z,k) &=& \frac{\T^+(k)}{4\pi\overline{k}}e^{-i(kz+\overline{k}\overline{z})}\overline{\mu^+(z,k)}, \\
\label{dbar_eq_m}
  \frac{\partial}{\partial\overline{k}}\, \mu^-(z,k) &=& \frac{\T^-(k)}{4\pi\overline{k}}e^{-i(k\overline{z}+\overline{k}z)}\overline{\mu^-(z,k)},
\end{eqnarray}
with a fixed parameter $z\in\R^2$ and requiring large $|k|$ asymptotics $\mu^{\pm}(z,\cdot)-1\in L^\infty\cap L^{r}(\C)$ for some $2<r<\infty$. 

Assuming that equations (\ref{dbar_eq_p}) and (\ref{dbar_eq_m}) have unique solutions with the appropriate asymptotic properties, set
$\psi^+(z,k):=e^{ikz}\mu^+(z,k)$ and $\psi^-(z,k):=e^{ik\overline{z}}\mu^-(z,k)$ and define formally 
\begin{eqnarray}\label{def:QQp}
    (\QQp\T^+)(z) &:=& \frac{i}{\pi^2}\dzb\int_{\C}\frac{\T^+(k)}{\overline{k}}\,e^{-ikz}\,\overline{\psi^+(z,k)} dk, \\
\label{def:QQm}
    (\QQm\T^-)(z) &:=& \frac{i}{\pi^2}\dz\int_{\C}\frac{\T^-(k)}{\overline{k}}\,e^{-ik\overline{z}}\,\overline{\psi^-(z,k)} dk, 
\end{eqnarray}
where $dk$ denotes Lebesgue measure:
$
\int_\C f(k)dk = \int_{\R^2} f(k_1,k_2) dk_1dk_2.
$

The inverse transform (\ref{def:QQp}) first appeared in \cite[formula (4.10)]{boiti}.  See the work
of Beals \& Coifman \cite{bealscoifman1,bealscoifman2,bealscoifman3,bealscoifman4}, Ablowitz \& Nachman
\cite{AblowitzNachman,NachmanAblowitz} and Henkin \& Novikov \cite{HenkinNovikov} for early references on the $\dbar$ method.

The functions $\T^{\pm}:\C\ra\C$ need to be ``well-behaved'' in order for equations (\ref{dbar_eq_p}) and (\ref{dbar_eq_m}) to be uniquely solvable, for the integrals in (\ref{def:QQp}) and (\ref{def:QQm}) to be convergent, and for the derivatives in (\ref{def:QQp}) and (\ref{def:QQm}) to make sense. According to \cite{LMS}, one example of good-enough behaviour is the following pair of assumptions:
$$
  \frac{\T^{\pm}(k)}{k}\in \SI(\C),\qquad   \frac{\T^{\pm}(k)}{\overline{k}}\in \SI(\C),
$$
where $\SI(\C)$ denotes rapidly decaying and infinitely smooth functions of Schwartz.

\section{Properties of solutions of the $\dbar$ equations} \label{sec:CGOprops}

\noindent
This section is concerned with solutions $\mu^\pm(z,k)$ of the $\dbar$ equations (\ref{dbar_eq_p}) and (\ref{dbar_eq_m}). The functions $\mu^\pm(z,0)$ are especially important as they will play a central role in later sections in the analysis of conductivity-type potentials.
We start by analyzing the decay in $|\mu^\pm(z,0)-1|$ when $|z|\ra\infty$.
\begin{lemma}\label{lemma:muz0_estimate}
Let $\T^\pm:\C\ra\C$ satisfy 
$$
  |\T^{\pm}(k)|\leq C|k|^2 \mbox{ for small }|k|,\qquad\frac{\T^{\pm}(k)}{k}\in \SI(\C),\qquad   \frac{\T^{\pm}(k)}{\overline{k}}\in \SI(\C).
$$
Fix $2<r<\infty$. For every $z\in\R^2$, let $\mu^\pm(z,k)$ be the unique solutions of the D-bar equations
(\ref{dbar_eq_p}) and (\ref{dbar_eq_m}) 
 with the asymptotic condition $\mu^\pm(z,\cdot)-1\in L^{r}\cap L^{\infty}$. 

Then the following estimate holds for all $z\in\R^2:$
\begin{equation}\label{mu0_est1}
  |\mu^\pm(z,0)-1| \leq C\langle z\rangle^{-1}.
\end{equation}
\end{lemma}
\begin{proof}
We prove estimate (\ref{mu0_est1}) for $\mu^+$ only; the proof for $\mu^-$ is analogous.

Denote $e^+_z(k):=\exp(i(kz+\overline{k}\overline{z}))$ and use equation
\begin{equation}\label{def:muz0}
  \mu^+(z,0):=\lim_{s\ra 0}\mu^+(z,s) = 1 + \frac{1}{4\pi^2}\int_{\R^2}\frac{\T^+(k)}{|k|^2}e^+_{-z}(k)\overline{\mu^+(z,k)}dk
\end{equation}
(given in \cite[formula (0.38)]{nachman}) to write
\begin{eqnarray}
  \mu^+(z,0)-1
  &=& \nonumber
  \frac{1}{4\pi^2}\int_\C\frac{\T^+(k)}{|k|^2}e^+_{-z}(k)\overline{\mu^+(z,k)}dk\\
  &=& \label{muz0_1}
  \frac{1}{4\pi^2}\int_\C\frac{\T^+(k)}{|k|^2}e^+_{-z}(k)dk\\
  && \label{muz0_2}
  +\frac{1}{4\pi^2}\int_\C\frac{\T^+(k)}{|k|^2}e^+_{-z}(k)(\overline{\mu^+(z,k)-1})dk.
\end{eqnarray}
Now term (\ref{muz0_1}) is finite because $|k|^{-2}\,\T^+(k)\in L^1(\C)$ by assumption. Term (\ref{muz0_2}) is finite because of H\"older's inequality and the inequalities $\||k|^{-2}\,\T^+(k)\|_{L^{r^\prime}(\C)}<\infty$ and $\|\mu^+(z,\,\cdot\,)-1\|_{L^r(\C)}\leq C<\infty$ with $C$ not depending on $z$ (see \cite[formula 4.1]{nachman}). Therefore
\begin{equation}\label{supnormmu0}
 \|\mu^+(z,0)\|_{L^\infty(\R^2)}<\infty.
\end{equation}

We can bound term (\ref{muz0_1}) as follows to study its behaviour when $|z|\ra \infty$:
\begin{eqnarray}
  && \nonumber
  |z|\left|\int_\C\frac{\T^+(k)}{|k|^2}e^+_{-z}(k)dk\right|\\
  &=& \nonumber
  \left|-iz\int_\C\frac{\T^+(k)}{|k|^2}e^+_{-z}(k)dk\right|\\
  &=& \nonumber
  \left|\int_\C\frac{\T^+(k)}{k\overline{k}}\frac{\partial e^+_{-z}(k)}{\partial k}dk\right|\\
  &\leq& \nonumber
  \left|\int_\C \frac{e^+_{-z}(k)}{\overline{k}}\frac{\partial}{\partial k}(\frac{\T^+(k)}{k})dk\right|\\
  && \label{vanishintterm}
  +\left|\int_\C \pi\delta_0(k)\frac{\T^+(k)}{k} e^+_{-z}(k)dk\right|\\
  &\leq& \nonumber
  \left\|\frac{1}{\overline{k}}\frac{\partial}{\partial k}(\frac{\T^+(k)}{k})\right\|_{L^1(\C)}\\
  &\leq& \nonumber
  \left\|\frac{1}{\overline{k}}\right\|_{L^1(|k|<1)}\left\|\frac{\partial}{\partial k}(\frac{\T^+(k)}{k})\right\|_{L^\infty(\C)}
  +\left\|\frac{\partial}{\partial k}(\frac{\T^+(k)}{k})\right\|_{L^1(\C)}\\
  &<& \label{muz0term1est}
 \infty.
\end{eqnarray}
Note that term (\ref{vanishintterm}) vanishes since $\lim_{k\ra 0}\T^+(k)/k=0$ by assumption .

Let us now bound term (\ref{muz0_2}).
\begin{eqnarray}
  &&\nonumber
  |z| \left|\int_{\C} \frac{\T^+(k)}{|k|^2} e^+_{-z}(k) (\overline{\mu^+(z,k)-1}) dk\right|\\
  &=& \nonumber
  \left|\int_{\C} \frac{\T^+(k)}{k\overline{k}} \frac{\partial e^+_{-z}(k)}{\partial k}(\overline{\mu^+(z,k)-1})dk\right|\\
  &\leq& \label{muz0term2est1}
  \left|\int_{\C} \frac{1}{\overline{k}} e^+_{-z}(k) (\overline{\mu^+(z,k)-1}) \frac{\partial}{\partial k}(\frac{\T^+(k)}{k}) dk\right|\\
  && \label{muz0term2est_0}
  +\left|\int_{\C} \pi\delta_0(k)\frac{\T^+(k)}{k} e^+_{-z}(k) (\overline{\mu^+(z,k)-1}) dk\right|\\
  &&\label{muz0term2est2}
  +\left|\int_{\C} \frac{\T^+(k)}{|k|^2}  e^+_{-z}(k)\overline{\dbar_k \mu^+(z,k)}dk\right|.
\end{eqnarray}
Similarly to the case of term (\ref{vanishintterm}), the term (\ref{muz0term2est_0}) vanishes. We can bound (\ref{muz0term2est1}) using H\"older inequality (we have $1<r^\prime<2$ since
$2<r<\infty$) and \cite[Lemma 3.3]{LMS}:
\begin{equation}\label{muxxxxxx}
  \left\|\frac{1}{\overline{k}}\frac{\partial}{\partial k}(\frac{\T^+(k)}{k})\right\|_{L^{r^\prime}(\C)} 
  \|\mu^+(z,\,\cdot\,)-1\|_{L^r(\C)} \leq C\langle z\rangle^{-1}.
\end{equation}
The $L^{r^\prime}$ norm in (\ref{muxxxxxx}) is finite because at infinity we have rapid decay and near the origin the derivative of $\T^+(k)/k\in {\mathcal S}(\C)$ is smooth and $|k|^{-r^\prime}$ is integrable as $r^\prime<2$.

To bound (\ref{muz0term2est2}) note that we get from the $\dbar$ equation
\begin{equation}\label{dbarcorollary}
  \frac{\partial}{\partial\overline{k}} \mu^+(z,k) = \frac{\T^+(k)}{4\pi\overline{k}} e^+_{-z}(k)(\overline{\mu^+(z,k)-1}) + \frac{\T^+(k)}{4\pi\overline{k}} e^+_{-z}(k).
\end{equation}
Estimate (\ref{muz0term2est2}) using (\ref{dbarcorollary}) and \cite[Lemma 3.3]{LMS}:
\begin{eqnarray}
  && \nonumber \!\!\!\!\!\!\!\!\!\!\!
  \left|\int_{\C} \frac{\T^+(k)}{|k|^2}  e^+_{-z}(k)\overline{\left(\frac{\T^+(k)}{4\pi\overline{k}} e^+_{-z}(k)(\overline{\mu^+(z,k)-1}) +
  \frac{\T^+(k)}{4\pi\overline{k}} e^+_{-z}(k)\right)}dk\right|\\
  &\leq& \label{muz0term2est3}
  \frac{1}{4\pi}\left|\int_{\C} \frac{|\T^+(k)|^2}{k|k|^2}(\mu^+(z,k)-1)dk\right| +
  \frac{1}{4\pi}\left|\int_{\C} \frac{|\T^+(k)|^2}{k|k|^2}dk\right|\\
  &\leq& \label{muz0term2est3B}
  \frac{1}{4\pi}\left\|\frac{1}{k} |\frac{\T^+(k)}{k}|^2\right\|_{L^{r^\prime}(\C)} \|\mu^+(z,k)-1\|_{L^{r}(\C)}  + C^\prime \\
  &\leq& \label{muxxx}
  C\langle z\rangle^{-1} + C^\prime.
\end{eqnarray}
Note that the second integral in (\ref{muz0term2est3}) and the $L^{r^\prime}$ norm in (\ref{muz0term2est3B}) are finite by the assumption that $|\T^+(k)|\leq C|k|^2$ for $k$ near zero. 

Combining (\ref{muz0_1}), (\ref{muz0_2}), (\ref{muz0term1est}), (\ref{muz0term2est1}), (\ref{muz0term2est2}), (\ref{muxxxxxx}) and (\ref{muxxx}) yields 
$$
  |z||\mu^+(z,0)-1|\leq C, 
$$
which together with (\ref{supnormmu0}) gives $  |\mu^+(z,0)-1| \leq C\langle z\rangle^{-1}$.
\end{proof}

Definition \ref{def:condtype} of conductivity-type potentials includes a derivative condition. The following lemma will be used in Section \ref{sec:preserveconductivity} for proving such a condition for the evolved potential.

\begin{lemma}\label{lemma:muz0_der_estimate}
Let $\T^\pm:\C\ra\C$ satisfy 
$$
 |\T^{\pm}(k)|\leq C|k|^2 \mbox{ for small }|k|,\qquad
\frac{\T^{\pm}(k)}{k}\in \SI(\C),\qquad   \frac{\T^{\pm}(k)}{\overline{k}}\in \SI(\C).
$$
Fix $2<r<\infty$. For every $z\in\R^2$, let $\mu^\pm(z,k)$ be the unique solutions of the D-bar equations
(\ref{dbar_eq_p}) and (\ref{dbar_eq_m}) with the asymptotic condition $\mu^\pm(z,\cdot)-1\in L^{r}\cap L^{\infty}$. Then
\begin{equation}\label{derivlemma}
  |\nabla \mu^\pm(z,0)| \leq C\langle z\rangle^{-2} \qquad \mbox{ for all }z\in\R^2.
\end{equation}
\end{lemma}
\begin{proof}
We prove estimate (\ref{derivlemma}) for $\mu^+$ only; the proof for $\mu^-$ is analogous.

We first prove the estimate $|\dzb \mu^+(z,0)| \leq C\langle z\rangle^{-2}$.
Differentiate (\ref{def:muz0}) to get
\begin{eqnarray}
  &&\nonumber
  \dzb\mu^+(z,0)\\
  &=& \nonumber
  \frac{1}{4\pi^2}\int_\C\left(-i\frac{\T^+(k)}{k}e^+_{-z}(k)\overline{\mu^+(z,k)}dk + \frac{\T^+(k)}{|k|^2}e^+_{-z}(k)\overline{\dz\mu^+(z,k)}\right)dk\\
  &=& \label{def:muz0derterm1}
  \frac{-i}{4\pi^2}\int_\C\frac{\T^+(k)}{k}e^+_{-z}(k)dk\\
  && \label{def:muz0derterm2}
  +\frac{1}{4\pi^2}\int_\C\Big(-i\frac{\T^+(k)}{k}e^+_{-z}(k)\overline{(\mu^+(z,k)-1)}dk \\
 &&\nonumber
  \qquad  \qquad  \qquad+ \frac{\T^+(k)}{|k|^2}e^+_{-z}(k)\overline{\dz\mu^+(z,k)}\Big)dk.
\end{eqnarray}
Now (\ref{def:muz0derterm1}) is rapidly decaying since it is the Fourier transform of a Schwartz function. Let us estimate (\ref{def:muz0derterm2}).
Integration by parts and applying equation (\ref{dbar_eq_p}) yields for $z\not=0$
\begin{eqnarray}
  && \label{continue_on}
  |z| \left|-i\int_{\C} \frac{\T^+(k)}{k} e^+_{-z}(k) (\overline{\mu^+(z,k)-1}) dk+
  \int_{\C}\frac{\T^+(k)}{|k|^2}e^+_{-z}(k)\overline{\dz\mu^+(z,k)}dk \right|\\
  &=& \nonumber
  \Big|-iz \int_{\C} \frac{\T^+(k)}{k} e^+_{-z}(k) (\overline{\mu^+(z,k)-1}) dk
  +i(-iz)\int_{\C}\frac{\T^+(k)}{|k|^2}e^+_{-z}(k)\overline{\dz\mu^+(z,k)}dk\Big|\\
  &=& \nonumber
  \left|\int_{\C} \frac{\T^+(k)}{k} \frac{\partial e^+_{-z}(k)}{\partial k}(\overline{\mu^+(z,k)-1})dk+
  i\int_{\C}\frac{\T^+(k)}{|k|^2}\frac{\partial e^+_{-z}(k)}{\partial k} \overline{\dz\mu^+(z,k)}dk\right|\\
  &=& \nonumber
  \Big|-\int_{\C}e^+_{-z}(k)\left[ \Big(\frac{\partial}{\partial k} \Big(\frac{\T^+(k)}{k}\Big)\Big)(\overline{\mu^+(z,k)-1})+\frac{\T^+(k)}{k}  \overline{\left(\frac{\partial}{\partial\overline{k}}\mu^+(z,k)\right)}\right] dk\\
  && \nonumber
  -\, i \int_{\C}  e^+_{-z}(k)\left[ \overline{\dz\mu^+(z,k)}\frac{\partial}{\partial k}\Big(\frac{\T^+(k)}{|k|^2}  \Big)
+\frac{\T^+(k)}{|k|^2}  \dzb\overline{\left(\frac{\partial}{\partial\overline{k}}\mu^+(z,k)\right)}\right]dk \Big|.
\end{eqnarray}
From equation (\ref{dbar_eq_p}) we get $\overline{\left(\frac{\partial}{\partial\overline{k}}\mu^+(z,k)\right)}=\frac{\overline{\T^+(k)}}{4\pi k} e^+_{z}(k)\mu^+(z,k)$ and 
\begin{eqnarray*}
 \dzb\overline{\left(\frac{\partial}{\partial\overline{k}}\mu^+(z,k)\right)} 
  &=&
  \dzb\big(\frac{\overline{\T^+(k)}}{4\pi k} e^+_{z}(k)\mu^+(z,k)\big) \\
  &=&
  i\overline{k}\,\frac{\overline{\T^+(k)}}{4\pi k} e^+_{z}(k)\mu^+(z,k) + 
  \frac{\overline{\T^+(k)}}{4\pi k} e^+_{z}(k)\dzb\mu^+(z,k).
\end{eqnarray*}
Hence, after cancellation of two terms containing $\frac{\overline{\T^+(k)}}{4\pi k} e^+_{z}(k)\mu^+(z,k)$, we see that (\ref{continue_on}) is bounded by
\begin{eqnarray}
 && \label{intterm1}
  \left|\int_{\C} e^+_{-z}(k)\Big(\frac{\partial}{\partial k} \Big(\frac{\T^+(k)}{k}\Big)\Big)(\overline{\mu^+(z,k)-1})dk\right| \\
  && \label{intterm3}
  + \left|\int_{\C} e^+_{-z}(k)\Big(\frac{\partial}{\partial k}\Big(\frac{\phantom{.}1\phantom{.}}{\overline{k}}\frac{\T^+(k)}{k}\Big)\Big) \overline{\dz\mu^+(z,k)} dk\right|\\
  &&\label{intterm4}
  + \left|\int_{\C}\frac{|\T^+(k)|^2}{4\pi k|k|^2}\,\dzb \mu^+(z,k)dk\right|.
\end{eqnarray}
We next estimate each of the terms (\ref{intterm1}),(\ref{intterm3}) and (\ref{intterm4}) separately.

H\"older's inequality and \cite[Lemma 3.3]{LMS} show that term (\ref{intterm1}) is bounded by
\begin{equation}\label{inttermEst1}
  \left\|\frac{\partial}{\partial k} \Big(\frac{\T^+(k)}{k}\Big)\right\|_{L^{r^\prime}(\C)}\|\mu^+(z,k)-1\|_{L^{r}(\C)}
  \leq
  C^\prime\langle z\rangle^{-1},
\end{equation}
where the $L^{r^\prime}$ norm is finite since by assumption $\T^+(k)/k\in {\mathcal S}(\C)$.
Term (\ref{intterm3}) can be estimated using H\"older's inequality and \cite[Lemma 3.4]{LMS}:
\begin{eqnarray}
  && \nonumber
  \left|\int_{\C} e^+_{-z}(k)\Big(\frac{\partial}{\partial k}\Big(\frac{\phantom{.}1\phantom{.}}{\overline{k}}\frac{\T^+(k)}{k}\Big)\Big) \overline{\dz\mu^+(z,k)} dk\right|\\
  &=&\nonumber
  \left|\pi\int_{\C} \delta_0(k)\frac{\T^+(k)}{k}e^+_{-z}(k) \overline{\dz\mu^+(z,k)} dk\right.\\
  &&\nonumber
  +\left.\int_{\C} \frac{e^+_{-z}(k)}{\overline{k}}\Big(\frac{\partial}{\partial k}\Big(\frac{\T^+(k)}{k}\Big)\Big) \overline{\dz\mu^+(z,k)} dk\right|\\
  &\leq&\label{inttermEst3}
  \left\|\frac{\phantom{.}1\phantom{.}}{\overline{k}}\frac{\partial}{\partial k}\Big(\frac{\T^+(k)}{k}\Big)\right\|_{L^{r^\prime}(\C)}
  \left\|\dz\mu^+(z,k)\right\|_{L^{r}(\C)}
  \leq
  C^\prime \langle z\rangle^{-1},
\end{eqnarray}
where the term containing $\delta_0(k)$ vanishes by the assumption since $\lim_{k\ra 0}\T^+(k)/k=0$. The finiteness of the $L^{r^\prime}$ norm in (\ref{inttermEst3}) is seen as in (\ref{muxxxxxx}).
Similarly, (\ref{intterm4}) is bounded by
\begin{eqnarray}
  \left|\int_{\C} \frac{|\T^+(k)|^2}{4\pi k|k|^2}\dzb\mu^+(z,k) dk\right|
    &\leq& \nonumber
    \frac{1}{4\pi}\left\|\frac{\phantom{.}1\phantom{.}}{k}\Big(\frac{\T^+(k)}{k}\Big)^2\right\|_{L^{r^\prime}(\C)}
  \left\|\dzb\mu^+(z,k)\right\|_{L^{r}(\C)}\\
    &\leq& \label{inttermEst4}
  C^\prime \langle z\rangle^{-1}.
\end{eqnarray}

Combining (\ref{inttermEst1}), (\ref{inttermEst3}) and (\ref{inttermEst4}) with (\ref{continue_on}) shows that (\ref{def:muz0derterm2}) is bounded by $C\langle z\rangle^{-2}$. Thus we may conclude that $|\dzb \mu^+(z,0)| \leq C\langle z\rangle^{-2}$. The proof for $|\dz \mu^+(z,0)| \leq C\langle z\rangle^{-2}$ is analogous and uses the assumption $\T^+(k)/\overline{k}\in \SI(\C)$. Together these two estimates yield (\ref{derivlemma}).
\end{proof}

\section{Radially symmetric initial data}\label{sec:radial}

\begin{theorem}
Let $q_0\in C^\infty_0(\R^2)$ be real-valued and of conductivity type in the sense of Definition \ref{def:condtype} with $\gamma\equiv 1$ outside the support of $q_0$. Furthermore, assume that $q_0$ is rotationally symmetric: $q_0(z)=q_0(|z|)$ for all $z\in\R^2$. 

Then $\qis_\tau$ defined by (\ref{qisdef}) is real-valued for all $\tau\geq 0$.
\end{theorem}
\begin{proof}
Theorem 3.3 of \cite{SMI2000} implies that $\T^+_0$ is rotationally symmetric and real-valued:
\begin{equation}\label{t0rotreal}
  \T^+_0(k) = \T^+_0(|k|),\qquad \overline{\T^+_0(k)}=\T^+_0(k).
\end{equation}
The proof of (\ref{t0rotreal}) is based on first using uniqueness of solutions to the Schr\"odinger equation (\ref{Schrode}) with the asymptotic condition (\ref{exp_asymp}) to show
\begin{eqnarray}
  \label{vaieq1}
  \mu^+_0(z,k) &=& \mu^+_0(e^{i\fii}z,e^{-i\fii}k), \\  
  \label{vaieq2}
  \mu^+_0(z,k) &=&  \overline{\mu^+_0(-\overline{z},\overline{k})},
\end{eqnarray}
for all $z\in\R^2$ and $k\in\C$ and $\fii\in\R$. Then $\T(k)=\T(e^{i\fii}k)$ and $\T(k)=\overline{\T(\overline{k})}$ by substituting (\ref{vaieq1}) and  (\ref{vaieq2}) to formula (\ref{tdef}), and (\ref{t0rotreal}) follows.

From the real-valuedness of $q_0$ and formula (\ref{pm_psi}) we know that
\begin{equation}\label{mu0relation}
  \overline{\mu^+_0(z,k)} = \mu^-_0(z,-\overline{k}),
\end{equation}
and we can calculate
\begin{eqnarray*}
  \T^-_0(-\overline{k})
  &=&
  \int_{\R^2} e^{-i((-\overline{k})\overline{z}+\overline{(-\overline{k})}z)} q_0(z)\overline{\mu^-_0(z,-\overline{k})}dz\\
  &=&
  \int_{\R^2} e^{i(\overline{k}\overline{z}+kz)} q_0(z)\overline{\overline{\mu^+_0(z,k)}}dz\\
  &=&
  \overline{\int_{\R^2} e^{-i(\overline{k}\overline{z}+kz)} q_0(z)\overline{\mu^+_0(z,k)}dz}\\
  &=& 
 \overline{\T^+_0(k)}.
\end{eqnarray*}
Applying (\ref{t0rotreal}) yields $\overline{\T^+_0(k)}=\T^+_0(k)=\T^+_0(-\overline{k})$, so we have
\begin{equation}\label{t0pm_identity}
  \T^+_0(k) = \T^-_0(k).
\end{equation}

Evolution of scattering data is defined by the same formula for $\T^+_\tau$ and $\T^-_\tau$:
\begin{eqnarray*}
  \T^+_\tau(k) &=&  e^{i\tau(k^3+\overline{k}^3)}\T^+_0(k),\\
  \T^-_\tau(k) &=&  e^{i\tau(k^3+\overline{k}^3)}\T^-_0(k),
\end{eqnarray*}
so we may use (\ref{t0pm_identity}) to conclude that $\T^+_\tau(k) = \T^-_\tau(k)$ for all $k\in\C$ and $\tau\geq 0$. In the rest of this section we denote 
\begin{equation}\label{t_taupm_identity}
  \T_\tau(k) := \T^+_\tau(k) = \T^-_\tau(k).
\end{equation}
Now (\ref{t_taupm_identity}) is a remarkable identity: we can write the (integral form) D-bar equations
\begin{eqnarray*}
  \mu^+_\tau(z,k) 
  &=&  
  1-\frac{1}{(4\pi)^2}\int_\C \frac{\T_\tau(\kp)}{\overline{\kp}(k-\kp)}e^{-i(\kp z+\overline{\kp}\overline{z})}\overline{\mu^+_\tau(z,\kp)}d\kp,\\
  \mu^-_\tau(z,k) 
  &=&  
  1-\frac{1}{(4\pi)^2}\int_\C \frac{\T_\tau(\kp)}{\overline{\kp}(k-\kp)}e^{-i(\kp\overline{z}+\overline{\kp}z)}\overline{\mu^-_\tau(z,\kp)}d\kp, 
\end{eqnarray*}
 and replacing $z$ by $\overline{z}$ in the latter D-bar equation gives the relation
\begin{equation}\label{mu_taupm_identity}
  \mu^+_\tau(z,k) = \mu^-_\tau(\overline{z},k),
\end{equation}
which, substituted into (\ref{def:QQm}), yields the identity
\begin{equation}\label{Q_taupm_identity}
  (\QQ^+ \T_\tau) (z) =  (\QQ^- \T_\tau) (\overline{z})
\end{equation}
for all $z\in\R^2$.

Our goal is to derive an equation connecting $ (\QQ^+ \T_\tau) (z)$ and $ (\QQ^+ \T_\tau) (\overline{z})$. For this we calculate 
\begin{eqnarray*}
  \overline{\mu^+_\tau(z,-\overline{k})} 
  &=&  
  1-\frac{1}{(4\pi)^2}\overline{\int_\C \frac{\T_\tau(\kp)}{\overline{\kp}(-\overline{k}-\kp)}
  e^{-i(\kp z+\overline{\kp}\overline{z})}\overline{\mu^+_\tau(z,\kp)}d\kp},\\
  &=&  
  1-\frac{1}{(4\pi)^2}\int_\C \frac{\T_0(\kp)e^{-i\tau((\kp)^3+(\overline{\kp})^3)}}{\kp(-k-\overline{\kp})}
  e^{i(\kp z+\overline{\kp}\overline{z})}\overline{\overline{\mu^+_\tau(z,\kp)}}d\kp,\\
  &=&  
  1-\frac{1}{(4\pi)^2}\int_\C \frac{\T_0(-\overline{\kp})e^{i\tau((\kp)^3+(\overline{\kp})^3)}}{-\overline{\kp}(-k+\kp)}
  e^{i((-\overline{\kp}) z-\kp\overline{z})}\overline{\overline{\mu^+_\tau(z,-\overline{\kp})}}d\kp,\\
  &=&  
  1-\frac{1}{(4\pi)^2}\int_\C \frac{\T_\tau(\kp)}{\overline{\kp}(k-\kp)}
  e^{-i((\overline{\kp}) z+\kp\overline{z})}\overline{\overline{\mu^+_\tau(z,-\overline{\kp})}}d\kp,
\end{eqnarray*}
so by the uniqueness of solutions to the D-bar equation we have
\begin{equation}\label{mupmtausym}
  \mu^-_\tau(z,k)=\overline{\mu^+_\tau(z,-\overline{k})}.
\end{equation}
Substituting (\ref{mupmtausym}) to the definition of $\QQ^+$ and using (\ref{t0rotreal}) yields
\begin{eqnarray}
  \overline{(\QQ^+\T_\tau)(z)}
  &=&   \nonumber
  \frac{-i}{\pi^2}\overline{\left(\dzb\int_\C \frac{\T_\tau(\kp)}{\overline{\kp}}
  e^{-i(\kp z+\overline{\kp}\overline{z})}\overline{\mu^+_\tau(z,\kp)}d\kp \right)}\\
  &=&   \nonumber
  \frac{-i}{\pi^2}\,\dz\int_\C \frac{\T_0(\kp)e^{-i\tau((\kp)^3+(\overline{\kp})^3)}}{\kp}
  e^{i(\kp z+\overline{\kp}\overline{z})}\overline{\overline{\mu^+_\tau(z,\kp)}}d\kp \\
  &=&   \nonumber
  \frac{-i}{\pi^2}\,\dz\int_\C \frac{\T_0(-\overline{\kp})e^{i\tau((\kp)^3+(\overline{\kp})^3)}}{-\overline{\kp}}
  e^{-i(\overline{\kp} z+\kp\overline{z})}\overline{\overline{\mu^+_\tau(z,-\overline{\kp})}}d\kp \\
  &=&   \nonumber
  \frac{i}{\pi^2}\,\dz\int_\C \frac{\T_\tau(\kp)}{\overline{\kp}}
  e^{-i(\overline{\kp} z+\kp\overline{z})}\overline{\mu^-_\tau(z,\kp)}d\kp \\
  &=& \label{QQQpmeq}
  (\QQ^-\T_\tau)(z).
\end{eqnarray}
Now a combination of (\ref{QQQpmeq}) and (\ref{Q_taupm_identity}) yields
$$
  (\QQ^+ \T_\tau) (z) =  \overline{(\QQ^+ \T_\tau) (\overline{z})}.
$$

Next we make use of the cubes appearing in the multiplier $\exp(i\tau(k^3+\overline{k}^3))$ of the evolving scattering data. Define $\fii=2\pi/3$ and note that $\exp(\pm i\fii)^3=1$. Denote the rotation of a complex number $z$ by angle $\fii$ by $z_\fii:=e^{i\fii}z$. The scattering data $\T_\tau$ has the following three-fold symmetry:
\begin{equation}\label{Tthreefoldsym}
  \T_ \tau(k_{\pm\fii}) = e^{i\tau((k_{\pm\fii})^3 + \overline{(k_{\pm\fii})}^3)}\T_0(k_{\pm\fii})
  = e^{i\tau(k^3+\overline{k}^3)}\T_0(k) = \T_\tau(k),
\end{equation}
where we used the rotational symmetry (\ref{t0rotreal}). Now we can compute
\begin{eqnarray*}
  \mu^+_\tau(z_\fii,k_{-\fii}) 
  &=&  
  1-\frac{1}{(4\pi)^2}\int_\C \frac{\T_\tau(\kp)}{\overline{\kp}(k_{-\fii}-\kp)}
  e^{-i(\kp z_\fii+\overline{\kp}\overline{z_\fii})}\overline{\mu^+_\tau(z_\fii,\kp)}d\kp,\\
  &=&  
  1-\frac{1}{(4\pi)^2}\int_\C \frac{\T_\tau(\kp_{-\fii})}{\overline{\kp_{-\fii}}(k_{-\fii}-\kp_{-\fii})}
  e^{-i(\kp_{-\fii} z_\fii+\overline{\kp_{-\fii}}\overline{z_\fii})}\overline{\mu^+_\tau(z_\fii,\kp_{-\fii})}d\kp,\\
  &=&  
  1-\frac{1}{(4\pi)^2}\int_\C \frac{\T_\tau(\kp)}{\overline{\kp}(k-\kp)}
  e^{-i(\kp z+\overline{\kp}\overline{z})}\overline{\mu^+_\tau(z_\fii,\kp_{-\fii})}d\kp,
\end{eqnarray*}
which by the uniqueness of the solution to the D-bar equation shows that $\mu^+_\tau(z,k)=\mu^+_\tau(z_\fii,k_{-\fii})$. The same argument works with $\fii$ replaced by $-\fii$, so we have the symmetry relation
\begin{equation}\label{muthreefoldsym}
  \mu^+_\tau(z,k)=\mu^+_\tau(z_{\pm\fii},k_{\mp\fii})
\end{equation}
for all $z\in\R^2$ and $k\in\C$. Denote the coordinate transformation of rotation by angle $\fii$ by $F_\fii(z):=z_\fii=e^{i\fii}z$. Then in the complex chain rule 
$$
  \dzb(f\circ F_\fii) = ((\dz f)\circ F_\fii)\cdot\dzb F_\fii + ((\dzb f)\circ F_\fii)\cdot \overline{\dz F_\fii}
$$
we have $\dzb F_\fii = \dzb (e^{i\fii}z)=0$ and $\dz F_\fii = \dz (e^{i\fii}z)=e^{i\fii}$, so 
\begin{equation}\label{chainrule!}
  ((\dzb f)\circ F_\fii) = e^{i\fii}\cdot \dzb(f\circ F_\fii).
\end{equation}
Applying (\ref{chainrule!}) to  the definition of $\QQ^+\T_\tau$ and using (\ref{muthreefoldsym}) and (\ref{Tthreefoldsym}) gives
\begin{eqnarray}
  (\QQ^+\T_\tau)(z_\fii)
  &=&   \nonumber
  \left(\dzb\Big(\frac{i}{\pi^2}\int_\C \frac{\T_\tau(\kp)}{\overline{\kp}}
  e^{-i(\kp z+\overline{\kp}\overline{z})}\overline{\mu^+_\tau(z,\kp)}d\kp \Big)\right)\circ F_\fii\\
  &=&   \nonumber
  e^{i\fii}\left(\frac{i}{\pi^2}\,\dzb\int_\C \frac{\T_\tau(\kp)}{\overline{\kp}}
  e^{-i(\kp z_\fii+\overline{\kp}\overline{z_\fii})}\overline{\mu^+_\tau(z_\fii,\kp)}d\kp\right) \\
  &=&   \nonumber
  \frac{ie^{i\fii}}{\pi^2}\,\dzb\int_\C \frac{\T_\tau(\kp)}{\overline{\kp}}
  e^{-i(\kp_\fii z+\overline{\kp_\fii}\overline{z})}\overline{\mu^+_\tau(z,\kp_\fii)}d\kp\\
  &=&   \nonumber
  \frac{ie^{i\fii}}{\pi^2}\,\dzb\int_\C \frac{\T_\tau(\kp_{-\fii})}{\overline{\kp_{-\fii}}}
  e^{-i(\kp z+\overline{\kp}\overline{z})}\overline{\mu^+_\tau(z,\kp)}d\kp\\
  &=&   \nonumber
  \frac{i}{\pi^2}\,\dzb\int_\C \frac{\T_\tau(\kp)}{\overline{\kp}}
  e^{-i(\kp z+\overline{\kp}\overline{z})}\overline{\mu^+_\tau(z,\kp)}d\kp\\
  &=& \label{QQQcube_eq}
  (\QQ^+\T_\tau)(z).
\end{eqnarray}

Now the combination of (\ref{QQQcube_eq}) and (\ref{Q_taupm_identity}) tells us that $\qis_\tau$ has two symmetries, three-fold and reflectional:
$$
  \qis_\tau(z) = \qis_\tau(z_\fii) = \qis_\tau(z_{-\fii}), \qquad \qis_\tau(z) = \overline{\qis_\tau(\overline{z})}.
$$
This implies the following for the real part:
\begin{equation}\label{Reqissym}
  \mbox{Re}\,\qis_\tau(z) = \mbox{Re}\,\qis_\tau(z_\fii) = \mbox{Re}\,\qis_\tau(z_{-\fii}), \qquad \mbox{Re}\,\qis_\tau(z) = \mbox{Re}\,\qis_\tau(\overline{z}).
\end{equation}
More importantly, we see that the imaginary part of $\qis_\tau$ satisfies
$$
  \mbox{Im}\,\qis_\tau(z) = \mbox{Im}\,\qis_\tau(z_\fii), \qquad \mbox{Im}\,\qis_\tau(z) = -\mbox{Im}\,\qis_\tau(\overline{z}),
$$
implying that 
$
  \mbox{Im}\,\qis_\tau(z)\equiv 0.
$

\end{proof}

\section{Preservation of conductivity type}\label{sec:preserveconductivity}

\noindent
We study the properties of the inverse scattering evolution $\qis_\tau$. We do not assume the symmetry $q_0(z)=q_0(|z|)$ in this section; instead, we just assume that $\qis_\tau$ stays real-valued for positive times $\tau>0$. 

We start by deriving partial differential equations connecting $\qis_\tau$ with the the solutions of the D-bar equations (\ref{dbar_eq_p}) and  (\ref{dbar_eq_m}).

\begin{lemma}\label{lemma:qdef}
Let $\T^\pm:\C\ra\C$ satisfy 
$$
\frac{\T^{\pm}(k)}{k}\in \SI(\C),\qquad   \frac{\T^{\pm}(k)}{\overline{k}}\in \SI(\C).
$$
Fix $2<r<\infty$. For every $z\in\R^2$, let $\mu^\pm(z,k)$ be the unique solutions of  equations
(\ref{dbar_eq_p}) and (\ref{dbar_eq_m}) with the asymptotic condition $\mu^\pm(z,\cdot)-1\in L^{r}\cap L^{\infty}(\C)$. 
Assume that $(\QQp\T^+)(z)$ and $(\QQm\T^-)(z)$, defined by (\ref{def:QQp}) and (\ref{def:QQm}), are real-valued functions.  

Then for any fixed $k\in\C\setminus 0$ we have
\begin{eqnarray} 
\label{dbx_p}
    (-\Delta - 4ik\dzb + \QQp\T^+)\mu^+(\,\cdot\,,k) &=& 0,\\
\label{dbx_m}
    (-\Delta - 4ik\dz + \QQm\T^-)\mu^-(\,\cdot\,,k) &=& 0.
 \end{eqnarray}
\end{lemma}

\begin{proof} The functions $(\QQ^\pm\T^\pm)(z)$ are well-defined by \cite[Thm 1.2]{LMS}.

 We will prove the lemma only for $\mu^+$, as the proof for $\mu^-$ is analogous.
Denote the solid Cauchy transform by
\begin{eqnarray}\label{def:C}
  \CC\fii(k) := \frac{1}{\pi }\int_{\C}\frac{\fii(k^\prime)}{k-k^\prime}dk^\prime,
\end{eqnarray}
where $dk^\prime$ denotes the Lebesgue measure. Note that $\CC$ and $\dkb$ are inverses of each other (modulo analytic
functions). Further, define a real-linear operator
\begin{eqnarray}\label{def:T}
  T_z\fii(z,k) := \frac{\T^+(k)}{4\pi\overline{k}}e_{-z}(k)\overline{\fii(z,k)}.
\end{eqnarray}
Nachman \cite{nachman} proved that the operator $[I-\CC T_z]:L^r(\C)\ra L^r(\C)$ is invertible and that $\CC T_z 1\in L^r(\C)$. Now the $\dbar$
equation can be written in the convenient form
\begin{eqnarray}\label{simpledbar}
  \mu^+ = 1 + \CC T_z\mu^+,
\end{eqnarray}
and the solution with appropriate asymptotics is given by
\begin{equation}
  \mu^+-1=[I-\CC T_z]^{-1}(\CC T_z 1).
\end{equation}

From the proof of \cite[Thm 1.1]{LMS} we know that the  commutator $\dzb(\dz+ik)$ with the operator of $\CC$ is given by
$$
  [\dzb(\dz+ik),\CC ] \fii =\frac{i}{\pi}\dzb\int_{\C}\fii(z,k^\prime)dk^\prime,
$$
and that the commutator of $\dzb(\dz+ik)$ with the operator $T_z$ vanishes:
$$
  [\dzb(\dz+ik),T_z] \fii(z,k) = 0.
$$
Applying the above commutator identities to (\ref{simpledbar}) yields
\begin{eqnarray}
  \dzb(\dz+ik)\mu^+
  &=& \nonumber
  \dzb(\dz+ik)(1 + \CC T_z\mu^+) = \dzb(\dz+ik)\CC T_z\mu^+ \\
  &=& \nonumber
  \CC T_z\,\dzb(\dz+ik)\mu^+ + \frac{i}{4\pi^2}\dzb\int_\C
\frac{\T^+(k^\prime)}{\overline{k}^\prime}e_{-z}(k^\prime)\overline{\mu^+(z,k^\prime)}dk^\prime\\
  &=& \label{commkaava}
  \CC T_z\,\dzb(\dz+ik)\mu^+ + \frac{\QQp\T^+}{4}.
\end{eqnarray}
Denote $f=\dzb(\dz+ik)\mu^+$. Now $\QQp\T^+$ is real-valued function of $z$ and does not depend on $k$, so from (\ref{commkaava})  we get
$$
  (I-\CC T_z)(f-\frac{\QQp\T^+}{4}) = \CC T_z(\frac{\QQp\T^+}{4}) = \frac{\QQp\T^+}{4}\CC
T_z\,1.
$$
Thus $f-\QQp\T^+/4 = (\QQp\T^+/4)(I-\CC T_z)^{-1}\CC T_z\,1 = (\QQp\T^+/4)(\mu^+-1)$. Finally
$$
  \dzb(\dz+ik)\mu^+ = f = \frac{\QQp\T^+}{4}(\mu^+-1)+\frac{\QQp\T^+}{4} =
\frac{\QQp\T^+}{4}\mu^+.
$$
\end{proof}

We are ready to prove that the inverse scattering evolution preserves conductivity type if it stays real-valued.
\begin{theorem} \label{thm:weak_cond_type}
Let $q_0\in L^p(\R^2)$ with $1<p<2$ be a real-valued potential with no exceptional points. Assume that the scattering data $\TT^\pm q_0=\T^\pm_0$ are well-defined and satisfy
$$
 |\T^{\pm}_0(k)|\leq C|k|^2 \mbox{ for small }|k|,\qquad
\frac{\T^{\pm}_0(k)}{k}\in \SI(\C),\qquad   \frac{\T^{\pm}_0(k)}{\overline{k}}\in \SI(\C),
$$
 and that $\qis_{\tau}=\QQp\big(e^{i\tau(k^3+\overline{k}^3)}\T^+_0(k)\big)$ is real-valued.  

Then $\qis_\tau$ is of conductivity type in the sense of Definition \ref{def:condtype} for all $\tau\geq 0$ and does not have exceptional points.
\end{theorem}
\begin{proof}
Set
\begin{equation}\label{ttauplus} 
  \T^+_\tau(k)=e^{i \tau(k^3+\overline{k}^3)}\,\T^+_0(k)
\end{equation}
for all $\tau\geq 0$ and note that $\T^+_\tau(k)/\overline{k}\in {\mathcal S}(\C)$ since by assumption $\T^+_0(k)/\overline{k}\in
{\mathcal S}(\C)$. Then by \cite[Thm 4.1]{nachman} we know that for
any fixed $z\in\R^2$ the $\dbar$ equation
\begin{equation}\label{dbar_tau}
  \frac{\partial}{\partial\overline{k}} \mu^+_\tau(z,k) = \frac{\T^+_\tau(k)}{4\pi\overline{k}}e^+_{-k}(z)\overline{\mu^+_\tau(z,k)}
\end{equation}
is uniquely solvable with the asymptotic condition $\mu^+_\tau(z,\cdot)-1\in L^{r}\cap L^{\infty}$ for some $r>2$. Furthermore, \cite[Thm
4.1]{nachman} also implies that $\mu^+_\tau(z,0):=\lim_{k\ra 0}\mu^+_\tau(z,k) \in L^{\infty}(\R^2)$ satisfies
\begin{equation}\label{0_awayfromzero}
  |\mu^+_\tau(z,0)| > 0 \qquad \mbox{ for all }z\in\R^2,
\end{equation}
and there is an $0<\epsilon<1$ such that we have the estimate
\begin{equation}\label{0_holder}
  \sup_z |\mu^+_\tau(z,0)-\mu^+_\tau(z,k)|\leq c|k|^\epsilon
\end{equation}
for $k$ near zero.

Now $\qis_\tau=\QQ^+\T^+_\tau$ is a well-defined continuous $L^p(\R^2)$ function for all $\tau\geq 0$ by \cite{LMS}. Furthermore, $\qis_\tau$ is real-valued by assumption, and $\T^+_\tau(k)/k\in {\mathcal S}(\C)$ is clear from combining (\ref{ttauplus}) with the assumption $\T^+_0(k)/k\in {\mathcal S}(\C)$. Thus we can apply Lemma \ref{lemma:qdef} to see that $\mu^+_\tau$ satisfies 
\begin{equation}\label{qLS2}
  (-\Delta - 4ik\dzb + \qis_\tau)\mu^+_\tau(\,\cdot\, ,k) = 0
\end{equation}
with any fixed $k\in\C\setminus 0$. Formula (\ref{0_holder}) and equation (\ref{qLS2}) imply in the sense of distributions
\begin{eqnarray*}
 \lim_{k\rightarrow 0} (\Delta + 4ik\dzb)\mu^+_\tau(z,k) &=& \Delta\mu^+_\tau(z,0),
\end{eqnarray*}
so $\Delta\mu^+_\tau(z,0)=\qis_\tau(z)\mu^+_\tau(z,0)$. Using (\ref{0_awayfromzero}) we can write
\begin{equation}\label{q_form}
  \qis_\tau(z)= \frac{\Delta\mu^+_\tau(z,0)}{\mu^+_\tau(z,0)}.
\end{equation}

Next we need to prove that $\mu^+_\tau(z,0)$ is real-valued. 
Denote 
\begin{equation}\label{ttauminus} 
 \T^-_\tau(k)=e^{i\tau (k^3+\overline{k}^3)}\, (\TT^- q_0)(k).
\end{equation}
Then we conclude as above that the D-bar equation
\begin{equation}\label{dbar_tauminus}
  \frac{\partial}{\partial\overline{k}} \mu^-_\tau(z,k) = \frac{\T^-_\tau(k)}{4\pi\overline{k}}e^-_{-k}(z)\overline{\mu^-_\tau(z,k)},
\end{equation}
where $e^-_{k}(z)=\exp(i(k\overline{z}+\overline{k}z))$, is uniquely solvable with the asymptotic condition $\mu^-_\tau(z,\cdot)-1\in L^{r}\cap L^{\infty}$ for some $r>2$. 

The real-valuedness of the initial data $q_0$ implies by (\ref{pm_tksym}) the symmetry $\T^-_0(k)=\overline{\T^+_0(-\overline{k})}$. Substituting this to (\ref{ttauplus}) and (\ref{ttauminus}) yields
$$
 \overline{\T^+_\tau(-\overline{k})}=\overline{e^{-i\tau (k^3+\overline{k}^3)}\, \T^+_\tau(-\overline{k})}=
e^{i\tau (k^3+\overline{k}^3)}\, \overline{\T^+_0(-\overline{k})} = \T^-_0(k).
$$
Calculate
\begin{eqnarray*}
\overline{\mu^+_{\tau}(z,-\overline{k})} 
&=& 
1 - \frac{1}{4\pi^2}\int \frac{\overline{\T^+_\tau(\kp)}}{\kp(\overline{-\overline{k}-\kp})}e^{i(\kp z +\overline{\kp}\overline{z})}\mu^+_{\tau}(z,\kp)d\kp \\
&=& 
1 - \frac{1}{4\pi^2}\int \frac{\overline{\T^+_\tau(-\overline{\kp})}}{(-\overline{\kp})(-k+\kp)}e^{ i((-\overline{\kp}) z -\kp\overline{z})}\mu^+_{\tau}(z,-\overline{\kp})d\kp \\
&=& 
1 - \frac{1}{4\pi^2}\int \frac{\T^-_\tau(\kp)}{\overline{\kp}(k-\kp)}e^{ -i(\overline{\kp} z +\kp\overline{z})}
\overline{\overline{\mu^+_{\tau}(z,-\overline{\kp})}}d\kp, 
\end{eqnarray*}
which, in view of uniqueness of solutions to (\ref{dbar_tauminus}), implies that
\begin{equation} \label{mus1}
\overline{\mu^+_{\tau}(z,-\overline{k})}=\mu_{\tau}^-(z,k).
\end{equation}
By (\ref{0_holder}) we have at the limit $k\ra 0$ the identity
\begin{equation} \label{mus}
\overline{\mu^+_{\tau}(z,0)}=\mu_{\tau}^-(z,0).
\end{equation}

A computation similar to (\ref{QQQpmeq}) shows that  $\QQp\T^+_{\tau} = \overline{\QQm\T^-_{\tau}}$. 
The assumption on real-valuedness of $\qis_\tau$ then implies $\qis_\tau=\QQp\T^+_{\tau} =\QQm\T^-_{\tau}$,  and by Lemma \ref{lemma:qdef} the functions  $\mu^{\pm}_{\tau}$ satisfy
\begin{eqnarray*}
    (-\Delta - 4ik\dzb + \qis_\tau)\mu^+_{\tau}(\cdot,k) &=& 0,\\
    (-\Delta - 4ik\dz + \qis_\tau)\mu^-_{\tau}(\cdot,k) &=& 0.
\end{eqnarray*}
 Taking the limit as $k\rightarrow 0$ in each of these equations implies
$$\mu^+_{\tau}(z,0) = \mu^-_{\tau}(z,0).$$
Combining this with \eqref{mus} implies
$$\overline{\mu^+_{\tau}(z,0)}=\mu_{\tau}^-(z,0)=\mu_{\tau}^+(z,0),$$
so $\mu_{\tau}^+(z,0)$ is real-valued.

Finally, by Lemma \ref{lemma:muz0_der_estimate} we have $|\nabla \mu^+_\tau(z,0)| \leq C\langle z\rangle^{-2},
$
so $\nabla \mu^+_\tau(z,0)\in L^p(\R^2)$ for all $1<p<2$. 
Hence $q_\tau$ is of  conductivity type in the sense of Definition \ref{def:condtype} with conductivity $\gamma:=\mu^+_\tau(z,0)^2$.  By \cite[Lemma 1.5]{nachman} $\qis_\tau$ has no exceptional points. 
\end{proof}

\section{Evolution of scattering data }\label{sec:scatevol}

\noindent
Assume that the initial potential $q_0\in L^p(\R^2)$ with $1<p<2$ is a real-valued potential with no exceptional points. Further, assume that the initial scattering data $\TT^\pm q_0=\T^\pm_0$ satisfies
\begin{equation}\label{Tassumptions}
 |\T^{\pm}_0(k)|\leq C|k|^2 \mbox{ for small }|k|,\qquad
\frac{\T^{\pm}_0(k)}{k}\in \SI(\C),\qquad   \frac{\T^{\pm}_0(k)}{\overline{k}}\in \SI(\C),
\end{equation}
and leads to a real-valued evolution 
$$
  \qis_\tau(z)=\QQp (e^{i\tau(k^3+\overline{k}^3)}\T^+_0(k))=\QQp \T^+_\tau.
$$ 
The aim of this section is to prove that the scattering data of $q_\tau$ evolves as expected; more precisely, that  $\TT^+(\QQp \T^+_\tau) = \T^+_\tau$. 

We remark that we do not assume the symmetry $q_0(z)=q_0(|z|)$ in this section.

The function $\QQp\T^+_\tau$ is constructed as explained in Section \ref{sec:Qmaps} using the unique solutions of the D-bar equation 
\begin{equation}\label{scatevol_dbar}
  \frac{\partial}{\partial\overline{k}} \,\mu^+_\tau(z,k) = 
 \frac{\T^+_\tau(k)}{4\pi\overline{k}}e^{-i(kz+\overline{k}\overline{z})}\overline{\mu^+_\tau(z,k)},
\end{equation}
with large $|k|$ asymptotics $\mu^+_\tau(z,\cdot)-1\in L^\infty\cap L^{r}(\C)$ for some $2<r<\infty$. By \cite{LMS} we know that 
$
  \qis_\tau=\QQp\T^+_\tau:\R^2\ra \C
$ 
is a well-defined continuous $L^p(\R^2)$ function with any $1<p<2$ and for all $\tau>0$. 

We wish to apply the nonlinear Fourier transform $\TT^+$ to the function $\qis_\tau$.
By Theorem \ref{thm:weak_cond_type} we know that $\qis_\tau $ is of conductivity type:
\begin{equation}\label{qForm}
  \qis_\tau(z)= \frac{\Delta\mu^+_\tau(z,0)}{\mu^+_\tau(z,0)},
\end{equation}
and does not have exceptional points.  Thus there exists for any $k\in\C\setminus 0$  a unique solution of the partial differential equation
\begin{equation}\label{rautalanka1}
  (-\Delta - 4ik\dzb + \qis_\tau(z))\mu^+_{\tau}(z,k) = 0
\end{equation}
with large $|z|$ asymptotics $\mu^+_{\tau}(\,\cdot\,,k)-1\in W^{1,\tilde{p}}(\R^2)$.

If we had additional decay in $\qis_\tau(z)$ as $|z|\ra\infty$, then we could construct $\TT^+\qis_\tau$ by the integral 
\begin{equation}\label{maybeScat}
 \int_{\R^2} e^{i(kz+\overline{k}\overline{z})}\qis_\tau(z)\mu^+_\tau(z,k)dz,
\end{equation}
which would be absolutely convergent. Furthermore, we could make use of \cite[Theorem 2.1]{nachman},  stating that the D-bar derivative $(\partial/\partial\overline{k}) \,\mu^+_\tau(z,k)$ is equal to
\begin{equation}\label{scatevol_dbar2}
 \frac{(\TT^+(\QQp\T^+_\tau))(k)}{4\pi\overline{k}}e^{-i(kz+\overline{k}\overline{z})}\overline{\mu^+_\tau(z,k)}.
\end{equation}
Furthermore, in view of Lemma \ref{lemma:qdef}, the unique solutions of equations (\ref{scatevol_dbar}) and (\ref{rautalanka1}) are the same functions. Therefore, comparing (\ref{scatevol_dbar}) and (\ref{scatevol_dbar2}) would yield the desired identity $\TT^+(\QQp \T^+_\tau) = \T^+_\tau$.

However, we do not have available any extra decay  in $\qis_\tau(z)$ as $|z|\ra\infty$; we just know that $|\qis_\tau(z)|\leq C\langle z\rangle^{-2}$. 
Therefore, at this point it is even unclear whether formula $\TT^+(\QQp\T^+_\tau)$ is well-defined.

To analyse $\TT^+(\QQp\T^+_\tau)$, we add and subtract the constant $1$ in (\ref{maybeScat}), write 
\begin{eqnarray}
 \TT^+ q_\tau 
 &=& \nonumber 
 \int_{\R^2} e^{i(kz+\overline{k}\overline{z})}\qis_\tau(z)(\mu^+_\tau(z,k)-1)dz+
 \int_{\R^2} e^{i(kz+\overline{k}\overline{z})}\qis_\tau(z)dz\\
  &=&\label{maybeScat2}
  T_1(k)+T_2(k),
\end{eqnarray}
and interpret $T_1$ and $T_2$ as follows.

For fixed $k\in\C\setminus 0$ the term $T_1(k)$ in (\ref{maybeScat2}) is bounded in absolute value by the H\"older inequality because $\mu^+_\tau(\,\cdot\,,k)-1\in L^{\tilde{p}}(\R^2)$ and $\qis_\tau\in L^{\tilde{p}^\prime}(\R^2)$. Note that $1<\tilde{p}^\prime<2$ since $2<\tilde{p}<\infty$. Furthermore, the norm $\|\mu^+_\tau(\,\cdot\,,k)-1\|_{L^{\tilde{p}}(\R^2)}$ depends continuously on $k$. This can be seen  as follows. 
The unique solution of the partial differential equation (\ref{rautalanka1}) with appropriate asymptotics is
given by
\begin{equation}\label{exp_constr}
  \mu^+_\tau(z,k)-1 = [I+g_k\ast(q_\tau\cdot\ )]^{-1}(g_k \ast q_\tau),
\end{equation}
as shown in \cite[p.82]{nachman}. Note that we have the estimates
\begin{eqnarray}
  \label{oper_est1}
\|g_k\ast q\|_{W^{1,\tilde{p}}(\R^2)} &\leq& C_k\|q\|_{L^p(\R^2)},\\
  \label{oper_est2}
  \|g_k\ast(q\cdot\ )\|_{L(W^{1,\tilde{p}}(\R^2))} &\leq& C_k^\prime\|q\|_{L^p(\R^2)}.
\end{eqnarray}
By \cite[formula (1.6)]{nachman} we have for any $h\in L^p(\R^2)$
$$
  g_k\ast h = -\frac{1}{4ik}[\dbar_z^{-1}-(\partial_z+ik)^{-1}\partial\dbar^{-1}]h,
$$
and by \cite[formula (1.2)]{nachman} we know that $\|(\partial_z+ik)^{-1}h\|_{L^{\tilde{p}}(\R^2)}\leq c\|h\|_{L^p(\R^2)}$ with $c$ independent of $k$. Hence the constants $C_k$ and $C_k^\prime$ in (\ref{oper_est1}) and (\ref{oper_est2}) and the norm $\|\mu^+_\tau(\,\cdot\,,k)-1\|_{L^{\tilde{p}}(\R^2)}$ depend continuously on $k$. Thus $T_1\in L^1_{\mbox{\tiny loc}}(\C\setminus 0)$.

The second term $T_2(k)$ in (\ref{maybeScat2}) can be interpreted as the Fourier transform of an $L^p(\R^2)$ function, the result belonging to $L^{p^\prime}(\R^2)$  by the Riesz-Thorin interpolation theorem. Here $p^\prime$ is  defined by $1=1/p+1/p^\prime$. 
Therefore $T_2\in L^1_{\mbox{\tiny loc}}(\C\setminus 0)$.

We can conclude from (\ref{maybeScat2}) that $\TT^+ q_\tau \in L^1_{\mbox{\tiny loc}}(\C\setminus 0)$, and consequently it can be interpreted as a distribution: $\TT^+ q_\tau \in \mathcal{D}^\prime(\C\setminus 0)$. 

Having established the existence of $\TT^+(\QQp \T^+_\tau)$ in the sense of distributions, we proceed to show that it equals $\T^+_\tau$. As mentioned above, we cannot apply \cite[Theorem 2.1]{nachman} to show that $(\partial/\partial\overline{k}) \,\mu^+_\tau(z,k)$ equals (\ref{scatevol_dbar2}) because there is not enough decay in $\qis_\tau(z)$ available as $z\ra \infty$. We overcome this problem by generalizing \cite[Theorem 2.1]{nachman}.
The crucial new technique is to approximate $\qis_\tau=\mu^+_\tau(z,0)^{-1}\Delta\mu^+_\tau(z,0)$ with a rapidly decaying conductivity-type potential in the norm of $L^p(\R^2)$ space so that the function $\mu^+_\tau(z,0)$ is approximated simultaneously. 

\begin{lemma}\label{Lemma_q_p} Let $q:\R^2\ra\C$ be a continuous, real-valued, conductivity type potential of the form $q(z)=\mu(z)^{-1}\Delta\mu(z)$. 
Assume that these estimates hold for all $z\in\R^2:$
\begin{eqnarray}
\label{MUESTQ}
|q(z)| &\leq& C\langle z\rangle^{-2},\\
\label{MUEST0}
  |\mu(z)| &\geq& c>0,\\
\label{MUEST1}
  |\mu(z)-1| &\leq& C\langle z\rangle^{-1},\\
\label{MUEST2}
  |\nabla \mu(z)| &\leq& C\langle z\rangle^{-2}.
\end{eqnarray}
Set $
  \fii(z) = \exp(-|z|^2) 
$
and $\fii_\eps(z) = \fii(\eps z)$ for all $\eps>0$. Define an approximation to $\mu$ by
\begin{equation}\label{muepsdef}
  \mu^{(\eps)}(z) := 1 + \fii_\eps(z)(\mu(z)-1),
\end{equation}
and an approximation to $q$ by
\begin{equation}\label{qapproxi}
  q^{(\eps)}(z) := \frac{\Delta \mu^{(\eps)}(z)}{\mu^{(\eps)}(z)}.
\end{equation}
 Then for any exponent $1<p<2$ we have 
\begin{equation}\label{Lpconvergence}
  \lim_{\eps\ra 0}\|q^{(\eps)}-q\|_{L^p(\R^2)} =  0.
\end{equation}
\end{lemma}

The proof of Lemma \ref{Lemma_q_p} is postponed to Appendix \ref{Appendix:Lemma}. 
We are ready to prove the main theorem of this section.

\begin{theorem}\label{mainThm2}
Let $q_0\in L^p(\R^2)$ with $1<p<2$ be a real-valued potential with no exceptional points. Assume that the scattering data $\TT^\pm q_0=\T^\pm_0$ satisfies
$$
  |\T^{\pm}_0(k)|\leq C|k|^2 \mbox{ for small }|k|,\qquad\frac{\T^{\pm}_0(k)}{k}\in \SI(\C),\qquad   \frac{\T^{\pm}_0(k)}{\overline{k}}\in \SI(\C).
$$
Assume that $q_\tau = \QQ^+\T^+_\tau = \QQ^+(e^{i\tau(k^3+\overline{k}^3)}\T^+_0(k))$ is real-valued.

Then $(\TT^+(\QQ^+\T^+_ \tau))(k)=\T^+_ \tau(k)$ for all $k\in\C\setminus 0$.
\end{theorem}
\begin{proof}
By \cite[Theorem 1.2]{LMS} we know that $q_\tau:\R^2\ra\C$ is a continuous function and satisfies $|q_\tau(z)|\leq C\langle z\rangle^{-2}$. Thus $q_\tau$ belongs to $L^p(\R^2)$ for any $1<p<2$. By Theorem \ref{thm:weak_cond_type} and \cite[Lemma 1.5]{nachman} we may conclude that $q_\tau$ does not have exceptional points and  we can write $q_\tau(z)= (\Delta\mu^+_\tau(z,0))/\mu^+_\tau(z,0)$.

Define for all $\eps>0$ the approximate potential
\begin{equation}\label{qapproxi2}
  q_\tau^{(\eps)}(z) := \frac{\Delta \mu_\tau^{(\eps)}(z,0)}{\mu_\tau^{(\eps)}(z,0)},
\end{equation}
where the function $\mu_\tau^{(\eps)}$ is given by
\begin{equation}\label{muepsdef2}
  \mu_\tau^{(\eps)}(z,0) := 1 + \fii_\eps(z)(\mu^+_\tau(z,0)-1).
\end{equation}
Now the approximate potential $q_\tau^{(\eps)}$ decays exponentially when $|z|\ra\infty$.
Thus, since $\mu^+_\tau(z,0)$ is real-valued, the function $q_\tau^{(\eps)}(z)$ satisfies the assumptions of \cite[Theorem 1.1]{nachman}. There are no exceptional points for $q_\tau^{(\eps)}(z)$ with any $\eps>0$, and we can define for all $k\in \C\setminus 0$
\begin{equation}\label{exp_constr_eps}
  \mu_\tau^{(\eps)}(z,k)-1 = [I+g_k\ast(q_\tau^{(\eps)}\cdot\ )]^{-1} (g_k \ast q_\tau^{(\eps)}).
\end{equation}
Furthermore, by \cite[Theorem 2.1]{nachman} the functions $\mu_\tau^{(\eps)}(z,k)$ satisfy the $\dbar$-equation
\begin{equation}\label{taudbar_eps}
  \frac{\partial}{\partial\overline{k}}\mu_\tau^{(\eps)}(z,k) = \frac{(\TT^+ (q_\tau^{(\eps)}))(k)}{4\pi\overline{k}}e^+_{-z}(k)\overline{\mu_\tau^{(\eps)}(z,k)}
\end{equation}
with the asymptotic condition $\mu_\tau^{(\eps)}(z,\cdot)-1\in L^r\cap L^\infty(\C)$.

How are the functions $\mu_\tau^{(\eps)}(z,k)$ related to the solutions $\mu^+_\tau(z,k)$? It is clear from (\ref{muepsdef2}) that $\mu_\tau^{(\eps)}(z,0)$ tends to $\mu^+_\tau(z,0)$ as $\eps\ra 0$, but how about nonzero $k$? Subtracting (\ref{exp_constr_eps}) from (\ref{exp_constr}) and using the resolvent equation 
\begin{eqnarray*}
   &&
  [I+g_k\ast(q_\tau\cdot\ )]^{-1}- [I+g_k\ast(q_\tau^{(\eps)}\cdot\ )]^{-1}\\
  &=&
  -[I+g_k\ast(q_\tau\cdot\ )]^{-1}[g_k\ast((q_\tau-q_\tau^{(\eps)})\cdot\ )] [I+g_k\ast(q_\tau^{(\eps)}\cdot\ )]^{-1}
\end{eqnarray*}
together with (\ref{oper_est1}) and (\ref{oper_est2}) yield
\begin{eqnarray}
  && \nonumber
  \|\mu_\tau^+(z,k)-\mu_\tau^{(\eps)}(z,k)\|_{W^{1,\tilde{p}}}\\
  &=& \nonumber
  \| [I+g_k\ast(q_\tau\cdot\ )]^{-1}(g_k \ast q_\tau)- [I+g_k\ast(q_\tau^{(\eps)}\cdot\ )]^{-1} (g_k \ast q_\tau^{(\eps)})\|_{W^{1,\tilde{p}}}\\
  &\leq& \nonumber
  \| [I+g_k\ast(q_\tau\cdot\ )]^{-1}(g_k \ast (q_\tau- q_\tau^{(\eps)}))\|_{W^{1,\tilde{p}}}\\
  && \nonumber
  +\|([I+g_k\ast(q_\tau\cdot\ )]^{-1}- [I+g_k\ast(q_\tau^{(\eps)}\cdot\ )]^{-1}) (g_k \ast q_\tau^{(\eps)})\|_{W^{1,\tilde{p}}}\\
  &\leq& \nonumber
  \| [I+g_k\ast(q_\tau\cdot\ )]^{-1}\|_{L(W^{1,\tilde{p}})}\|(g_k \ast (q_\tau- q_\tau^{(\eps)}))\|_{W^{1,\tilde{p}}}\\
  && \nonumber
  +\|([I+g_k\ast(q_\tau\cdot\ )]^{-1}[g_k\ast((q_\tau-q_\tau^{(\eps)})\cdot\ )] [I+g_k\ast(q_\tau^{(\eps)}\cdot\ )]^{-1}) (g_k \ast q_\tau^{(\eps)})\|_{W^{1,\tilde{p}}}\\
  &\leq& \nonumber
  C_k^\prime\|(g_k \ast (q_\tau- q_\tau^{(\eps)}))\|_{W^{1,\tilde{p}}}\\
  && \nonumber
  +C_k^{\prime\prime}\|[I+g_k\ast(q_\tau^{(\eps)}\cdot\ )]^{-1}\|_{L(W^{1,\tilde{p}})} 
  \|g_k\ast((q_\tau-q_\tau^{(\eps)})\cdot\ )\|_{L(W^{1,\tilde{p}})}  
  \|g_k \ast q_\tau^{(\eps)}\|_{W^{1,\tilde{p}}}\\
  &\leq& \nonumber
  C_k^{\prime\prime\prime}\Big(1+\|[I+g_k\ast(q_\tau^{(\eps)}\cdot\ )]^{-1}\|_{L(W^{1,\tilde{p}})} 
  \|g_k \ast q_\tau^{(\eps)}\|_{W^{1,\tilde{p}}}\Big)\|q_\tau- q_\tau^{(\eps)}\|_{L^p}.
\end{eqnarray}
From Lemma \ref{Lemma_q_p} we know that
\begin{equation}\label{Lpconvergencetaas}
  \lim_{\eps\ra 0}\|q_\tau^{(\eps)}-q_\tau\|_{L^p(\R^2)} =  0.
\end{equation}
Therefore, recalling that the constants $C_k$ and $C_k^\prime$ in (\ref{oper_est1}) and (\ref{oper_est2}) depend continuously on $k$, we can conclude that 
\begin{equation}\label{mukepsconv}
  \lim_{\eps\ra 0} \|\mu_\tau^+(\,\cdot\,,k)-\mu_\tau^{(\eps)}(\,\cdot\,,k)\|_{W^{1,\tilde{p}}(\R^2)} = 0,
\end{equation}
where the convergence is uniform for $k\in K\subset\C\setminus 0$ with any compact $K$. We remark that since $W^{1,\tilde{p}}(\R^2)$ functions are continuous by the Sobolev imbedding theorem, equation (\ref{mukepsconv}) implies uniform point-wise convergence as well.

Now we can use (\ref{maybeScat2}) and (\ref{mukepsconv}) to calculate
\begin{eqnarray}
  \lim_{\eps\ra 0} \TT^+(q_\tau^{(\eps)})
  &=& \nonumber
  \lim_{\eps\ra 0} \int_{\R^2} e^{i(kz+\overline{k}\overline{z})} q_\tau^{(\eps)}(z)\mu_\tau^{(\eps)}(z,k)dz\\
  &=& \nonumber
  \lim_{\eps\ra 0} \int_{\R^2} e^{i(kz+\overline{k}\overline{z})} q_\tau^{(\eps)}(z)(\mu_\tau^{(\eps)}(z,k)-1)dz\\
  && \nonumber
  +\lim_{\eps\ra 0} \int_{\R^2} e^{i(kz+\overline{k}\overline{z})} q_\tau^{(\eps)}(z)dz\\
  &=:& \nonumber
  \lim_{\eps\ra 0} T_1^{(\eps)}(k) +  \lim_{\eps\ra 0} T_2^{(\eps)}(k)\\
  &=& \nonumber
   T_1(k) + T_2(k)\\
  &=& \nonumber
 \TT^+(q_\tau)\\
  &=& \label{qscatpreka}
 \TT^+(\QQ^+\T^+_\tau),
\end{eqnarray}
Note that $\lim_{\eps\ra 0} T_1^{(\eps)}=T_1$ in the topology of $L^1_{\mbox{\tiny loc}}(\C\setminus 0)$ because of (\ref{mukepsconv}) and $\lim_{\eps\ra 0} q_\tau^{(\eps)}=q_\tau$ in $L^{\tilde{p}^\prime}(\R^2)$ by Lemma \ref{Lemma_q_p}. Also,  $\lim_{\eps\ra 0} T_2^{(\eps)}=T_2$ in the topology of $L^{p^\prime}(\R^2)$ because $\lim_{\eps\ra 0} q_\tau^{(\eps)}=q_\tau$ in $L^p(\R^2)$ by Lemma \ref{Lemma_q_p}. Thus the convergence in (\ref{qscatpreka}) happens in particular in $L^1_{\mbox{\tiny loc}}(\C\setminus 0)$, and consequently in $\mathcal{D}^\prime(\C\setminus 0)$.

We know by construction that the unique solutions $\mu^+_\tau$ satisfy
\begin{equation}\label{scatevol_dbar_again!}
  \frac{\partial}{\partial\overline{k}} \,\mu^+_\tau(z,k) = 
 \frac{\T^+_\tau(k)}{4\pi\overline{k}}e^{-i(kz+\overline{k}\overline{z})}\overline{\mu^+_\tau(z,k)},
\end{equation}
The above analysis shows that equation (\ref{taudbar_eps}) converges in the sense of distributions $\mathcal{D}^{\prime}(\C\setminus0)$ to 
\begin{equation}\label{scatevol_dbar_second!}
  \frac{\partial}{\partial\overline{k}} \,\mu^+_\tau(z,k) = 
 \frac{(\TT^+(\QQ^+\T^+_\tau))(k)}{4\pi\overline{k}}e^{-i(kz+\overline{k}\overline{z})}\overline{\mu^+_\tau(z,k)}
\end{equation}
as $\eps\ra 0$. Thus right hand sides of the $\dbar$ equations (\ref{scatevol_dbar_again!}) and (\ref{scatevol_dbar_second!}) must be the same elements of $\mathcal{D}^{\prime}(\C\setminus0)$:
$$
\frac{(\TT^+(\QQ^+\T^+_\tau))(k)}{4\pi\overline{k}}e^{-i(kz+\overline{k}\overline{z})}\overline{\mu^+_\tau(z,k)}=
  \frac{e^{i\tau(k^3+\overline{k}^3)}\T^+_0(k)}{4\pi\overline{k}}e^{-i(kz+\overline{k}\overline{z})}\overline{\mu^+_\tau(z,k)}.
$$
Since $e^{-i(kz+\overline{k}\overline{z})}$ never vanishes, we get for $k\not=0$
$$
 (\TT^+(\QQ^+\T^+_\tau))(k)\overline{\mu^+_\tau(z,k)} =
  \T^+_\tau(k)\overline{\mu^+_\tau(z,k)}.
$$
Now the function $\mu^+_\tau(\,\cdot\,,k)$ cannot be identically zero since $\mu^+_\tau(\,\cdot\,,k)-1\in W^{1,\tilde{p}}(\R^2)$. Furthermore, $\T^+_\tau$ is a smooth function and $\TT^+(\QQ^+\T^+_\tau)\in\mathcal{D}^{\prime}(\C)$. We may conclude that $(\TT^+(\QQ^+\T^+_ \tau))(k)=\T^+_ \tau(k)$ for all $k\not=0$.

\end{proof}

\section{Conclusion}\label{sec:concl}

\noindent
The following Corollary is the main result of this paper. 
\begin{cor}\label{conclcor}
Let $q_0\in C^\infty_0(\R^2)$ be a real-valued, smooth, compactly supported con\-ducti\-vity-type potential (in the sense of Definition \ref{def:condtype}) with $\gamma\equiv 1$ outside supp$(q_0)$. Assume the rotational symmetry $q_0(z)=q_0(|z|)$. Denote $\T^+_\tau(k)=e^{i\tau(k^3+\overline{k}^3)}(\TT^+q_0)(k)$.
Then $\qis_\tau:=\QQp\T^+_\tau$ is for all $\tau>0$ a real-valued, continuous, conductivity-type potential in $L^p(\R^2)$ with any $1<p<2$ satisfying the following estimate: $|\qis_\tau(z)|\leq C\langle z\rangle^{-2}$.
Moreover, $\qis_\tau$ has no exceptional points and $(\TT^+(\QQ^+\T^+_\tau))(k)=\T^+_\tau(k)$ for all $k\not=0$.
\end{cor} 
\begin{proof}
We know from \cite[Thm 3.1]{SMI2000} and \cite[Thms 2.1 and 2.2]{LMS} that
$$
  |\T^{\pm}_0(k)|\leq C|k|^2 \mbox{ for small }|k|,\qquad\frac{\T^{\pm}_0(k)}{k}\in \SI(\C),\qquad   \frac{\T^{\pm}_0(k)}{\overline{k}}\in \SI(\C).
$$
The evolving potential is a well-defined continuous function $\qis_\tau:\R^2\ra\C$ satisfying $|\qis_\tau(z)|\leq C\langle z\rangle^{-2}$ by \cite[Thm 1.2]{LMS}.

The analysis in Section \ref{sec:radial} above shows that $\qis_\tau$ is real-valued. Then the assumptions of Lemma \ref{lemma:qdef} and Theorem \ref{thm:weak_cond_type} are satisfied and we may conclude that $\qis_\tau$ is of conductivity type in the sense of Definition \ref{def:condtype} for all $\tau\geq 0$ and that $\qis_\tau$ has no exceptional points. 
Furthermore, the assumptions of Theorem \ref{mainThm2} are fulfilled, and so $(\TT^+(\QQ^+\T^+_\tau))(k)=\T^+_\tau(k)$ for all $k\not=0$.
\end{proof}

Our results do in fact apply to more general initial data than the rotationally symmetric cases of Corollary \ref{conclcor}. Namely, Sections \ref{sec:preserveconductivity} and \ref{sec:scatevol} only assume that $q_0$ has nicely behaving scattering data and that the inverse scattering evolution stays real-valued. (The latter assumption is natural: the right-hand side of the Novikov-Veselov equation (\ref{evolution}) is real-valued, so real-valuedness of $q_0$ implies that of $\qnv_{\tau}$ as well. If $\qis_\tau$ had nonzero imaginary part, the identity $\qnv_{\tau} =\qis_{\tau}$ could not hold.)

Actually, we can already describe a class of nonsymmetric initial data for the inverse scattering evolution. Let $\qis_{\tau}$ be an evolution with rotationally symmetric initial data satisfying the assumptions of Corollary \ref{conclcor}, and fix $\tau^\prime>0$. Define $\widetilde{q}_0(z):=\qis_{\tau^\prime}(z)$ and note that the following diagram is well-defined:
$$
\begin{picture}(380,130)
\thinlines
 \put(30,110){\vector(1,0){130}}
 \put(200,110){\vector(1,0){130}}
 \put(0,107){$\T^+_0(k)$}
 \put(165,107){$\T^+_{\tau^\prime}(k)$}
 \put(335,107){$\T^+_{\tau^\prime+\tau}(k)$}
 \put(55,115){$\exp(i \tau^\prime(k^3+\overline{k}^3))\cdot$}
 \put(225,115){$\exp(i \tau(k^3+\overline{k}^3))\cdot$}
 \put(12,51){\vector(0,1){47}}
 \put(-7,72){$\TT^+$}
 \put(0,35){$q_0(z)$}
 \put(156,72){$\TT^+$}
 \put(175,51){\vector(0,1){47}}
 \put(181,97){\vector(0,-1){47}}
 \put(184,72){$\QQp$}
 \put(332,72){$\TT^+$}
 \put(351,51){\vector(0,1){47}}
 \put(357,97){\vector(0,-1){47}}
 \put(360,72){$\QQp$}
 \put(165,35){$\qis_{\tau^\prime}(z)$}
 \put(335,35){$\qis_{\tau^\prime+\tau}(z)$}
 \put(175,25){\rotatebox{-90}{$=$}}
 \put(351,25){\rotatebox{-90}{$=$}}
 \put(165,0){$\widetilde{q}_0(z)$}
 \put(340,0){$\widetilde{q}_{\tau}(z)$}
\end{picture}
$$
Now $\widetilde{\qis}_0(z)$ is valid initial data (if we substitute $\TT^+(\widetilde{\qis}_0(z))(0):=\TT^+(\qis_{\tau^\prime}(z))(0)=0$ making the initial scattering data smooth) and leads to a real-valued inverse scattering evolution $\widetilde{q}_{\tau}$. But what do we know about symmetries of $\widetilde{q}_{\tau}$? The analysis  in Section \ref{sec:radial} implies that 
$$
  \widetilde{q}_{\tau}(z) = \widetilde{q}_{\tau}(z_\fii) = \widetilde{q}_{\tau}(z_{-\fii}), \qquad \widetilde{q}_{\tau}(z) = \overline{\widetilde{q}_{\tau}(\overline{z})}
$$
for $\fii=2\pi/3$. But our theoretical results do not rule out the possibility of rotational symmetry of $\widetilde{q}_{\tau}$. However, in Part II of this paper we compute $\qis_{\tau}(z)$ numerically for several rotationally symmetric initial data and observe that for $\tau^\prime>0$ we have in general $\qis_{\tau^\prime}(z)\not=\qis_{\tau^\prime}(|z|)$; this is verified numerically beyond doubt. (In addition, $\qis_{\tau}$ appears not to be compactly supported, but the numerical evidence is non-conclusive due to the finite computational domain.) 

Finally, we mention that the numerical evidence presented in Part II of this paper strongly suggests that $\qnv_{\tau} =\qis_{\tau}$ for evolutions with initial data satisfying the assumptions of Corollary \ref{conclcor}.

\section*{Acknowledgements}

\noindent
The work of ML and SS was supported by the Academy of Finland (Finnish Centre of Excellence
in Inverse Problems Research 2006-2011, decision number 213476). In addition, SS was funded in part by the Grant-in-Aid for JSPS Fellows (No. 0002757) of the Japan Society for the Promotion of Science. This project  was partially conducted at  the Mathematical Sciences Research Institute,  Berkeley, whose hospitality is gratefully acknowledged.  

\appendix

\section{Connections to previous work}\label{appendix:connections}

\subsection{Remarks on the results of Boiti, Leon, Manna and Pempinelli} \label{appendix_Boiti}

It was shown in \cite{boiti}  that if
a function $\mathbf{p}(z,\tau)$ does not have exceptional points and
evolves according to the evolution equation
\begin{equation}\label{BLMPeq}
  \frac{\partial \mathbf{p}}{\partial \tau} = -a_0\dz^3 \mathbf{p} - a_0\dzb^3 \mathbf{p} + 3a_0\dz(\mathbf{p}\,\dz\dzb^{-1}\mathbf{p}) + 3a_0\dzb(\mathbf{p}\,\dzb\dz^{-1}\mathbf{p}),
\end{equation}
where $a_0\in\R$ is a constant, then the scattering data
\begin{equation}\label{def:F1}
  F_1(k) = \frac{1}{4\pi}\int_{\R^2}e^{i\overline{k}\overline{x}}\mathbf{p}(z,\tau)\phi(z,k)dx
\end{equation}
evolves in $\tau$ as follows:
\begin{equation}\label{F1evol}
  \frac{\partial}{\partial \tau}F_1(k) = i(a_0k^3 + \overline{a_0}\overline{k}^3)F_1(k).
\end{equation}
The function $\phi(z,k)$ appearing in (\ref{def:F1}) is a solution of the Schr\"odinger equation
\begin{equation}\label{boitiShrode}
  (\dz\dzb - \mathbf{p})\phi(\,\cdot\,,k) = 0,\qquad \phi(z,k) \sim e^{ikz}.
\end{equation}
Formula (\ref{BLMPeq}) corresponds to \cite[formula (2.12)]{boiti},
formula (\ref{def:F1}) corresponds to \cite[formula (3.26)]{boiti},
formula (\ref{F1evol}) corresponds to \cite[formula (5.7)]{boiti}, and
formula (\ref{boitiShrode}) corresponds to \cite[formulae
(2.3),(3.1),(3.4)]{boiti}.

Set $\qis_\tau(z)=\mathbf{p}(z/2,\tau)$. Then
equation (\ref{BLMPeq}) takes the form
\begin{equation}\label{NVqfromp}
   \frac{\partial \qis_\tau}{\partial \tau}  = -8a_0\dz^3 \qis_\tau - 8a_0\dzb^3 \qis_\tau + 6a_0\dz(\qis_\tau\,\dz\dzb^{-1}\qis_\tau) + 6a_0\dzb(\qis_\tau\,\dzb\dz^{-1}\qis_\tau).
\end{equation}
Further, noting that $\dz\dzb=\frac 14\Delta$ and comparing
(\ref{Schrode}) and (\ref{exp_asymp}) with (\ref{boitiShrode}) shows
that
\begin{equation}\label{psiphiconversion}
  \phi(z,k) = \psi^+(2z,k/2).
\end{equation}
Thus we can use (\ref{psiphiconversion}) and  (\ref{def:F1}) to compute
\begin{eqnarray}
  F_1(k)
  &=&\nonumber
  \frac{1}{8\pi}\int_{\R^2}e^{i\overline{k}\overline{(z/2)}}\mathbf{p}(z/2,\tau)\phi(z/2,k)dz\\
  &=&\nonumber
  \frac{1}{8\pi}\int_{\R^2}e^{i\overline{(k/2)}\overline{z}}\qis_\tau(z)\psi^+(z,k/2)dz\\
  &=&\label{F1integralform}
  \frac{1}{8\pi}\T^+_\tau(k/2).
\end{eqnarray}
A combination of (\ref{F1evol}) and (\ref{F1integralform}) then yields
\begin{equation}\label{a0tevol}
  \frac{\partial}{\partial \tau}\T^+_\tau(k) = i(8a_0k^3+8\overline{a_0}\overline{k}^3)\T^+_\tau(k).
\end{equation}
The desired evolution of scattering data is achieved by the choice $a_0=1/8$. Then equation
(\ref{NVqfromp}) takes exactly the form (\ref{evolution}).

\subsection{Remarks on the results of Tsai}\label{appendix_Tsai}


It would be tempting to follow Tsai's proof in \cite{tsai3} to show that the inverse scattering evolution $q_\tau$ actually coincides with $\qnv_{\tau}$ in diagram (\ref{diagram}). However, the class of initial data used in \cite{tsai3} excludes conductivity-type potentials, which in turn are the only known initial data with no exceptional points. The specific problem with Tsai's proof is the requirement $m_+(x,0)=0$ on the line following equation (3.12) in \cite{tsai3}. In our notation this would mean $\mu^+(z,0)=0$ which never holds for conductivity-type potentials because $\mu^+(z,0)=\sqrt{\gamma(z)}\geq c>0$. Finding such assumptions on $q_0$ that $q_\tau=\qnv_{\tau}$ in diagram (\ref{diagram}) remains an open theoretical problem. 

Tsai gives in \cite{tsai,tsai3} a formal derivation of a hierarchy of evolution equations (parametrized by $n=1,3,5,\dots$) using the maps
$\TT^+: q\mapsto \T^+$ and $\QQ^+$ in the following inverse scattering scheme:
$$
  \begin{CD}
  \T_0^+ @>\exp(-i^n (k^n+\overline{k}^n)\tau)\cdot>>\T_\tau^+\\
  @A{\TT^+}AA @VV{\QQ^+}V\\
  q_0  @>>\mbox{\small \ }> q_\tau
  \end{CD}
$$
The non-periodic version of the Novikov-Veselov equation (\ref{evolution}) appears as the case $n=3$. We remark that all the results in this paper hold for the cases $n>3$ as well. We just need to replace the angle $\fii=2\pi/3$ by $\fii=2\pi/n$ in Section \ref{sec:radial}.

\section{Proof of Lemma \ref{Lemma_q_p}}\label{Appendix:Lemma}

Clearly $\qis\in L^p(\R^2)$ with any $1<p<2$. Note that assumption (\ref{MUEST0}) and formula (\ref{muepsdef}) imply (for the same constant $c$, independent of $\eps$)
\begin{equation}\label{muepspos}
  0<c\leq |\mu^{(\eps)}(z)|,
\end{equation}
so there is no division by zero in the definition (\ref{qapproxi}). Write
$$
  q^{(\eps)}(z)-q(z) = \Delta\mu(z)\left(\frac{1}{\mu^{(\eps)}(z)} -  \frac{1}{\mu(z)}\right)
  -\Big(\Delta\mu(z)-\Delta \mu^{(\eps)}(z)\Big)\frac{1}{\mu^{(\eps)}(z)}.
$$
Then (\ref{Lpconvergence}) follows from the triangle inequality and (\ref{muepspos}) if we prove the following
two equations:
\begin{eqnarray}
  \label{Lpconvergence2}
  &&\lim_{\eps\ra 0}\left\|\Delta\mu\left(\frac{1}{\mu^{(\eps)}} -  \frac{1}{\mu}\right)\right\|_{L^p(\R^2)} = 0, \\
  \nonumber\\
  \label{Lpconvergence3}
  &&\lim_{\eps\ra 0}\|\Delta\mu-\Delta \mu^{(\eps)}\|_{L^p(\R^2)} = 0.
\end{eqnarray}

Let us prove (\ref{Lpconvergence3}) first. Calculate
\begin{equation}\label{deltamuformula}
  \Delta \mu^{(\eps)} = (\mu-1)\Delta\fii_\eps + 2\nabla\fii_\eps\cdot\nabla\mu + \fii_\eps\Delta\mu,
\end{equation}
and further, with the notation $(x,y)=z=x+iy$,
\begin{eqnarray}
  \frac{\partial \fii_\eps(z)}{\partial x}
  &=& \label{fiider1}
  -2\eps^2 x \fii_\eps(z),\\
  \frac{\partial \fii_\eps(z)}{\partial y}
  &=& \label{fiider1y}
  -2\eps^2 y \fii_\eps(z),\\
  \Delta \fii_\eps(z)
  &=& \label{fiider2}
  4\eps^2(\eps^2 x^2 +\eps^2 y^2 - 1) \fii_\eps(z).
\end{eqnarray}
Direct computation shows that for any $s\geq 0$ we have
\begin{eqnarray}\label{fiiepspnorm}
  \big\| |z|^s\fii_\eps(z)\big\|_p 
  &=& \nonumber
  \left(\int_{\R^2}|z|^{ps}\fii(\eps z)^p dz \right)^{1/p} \\
  &=& \nonumber
  \left(\int_{\R^2}|w|^{ps}\eps^{-ps}\fii(w)^p \frac{dw}{\eps^2} \right)^{1/p} \\
  &=& \label{Lpapu}
  \eps^{-s-2/p}\big\| |w|^s\fii(w)\big\|_p.
\end{eqnarray}
Here and below we use the sorthand notation $\|\,\cdot\,\|_{L^p(\R^2)}=\|\,\cdot\,\|_p$. 

By (\ref{deltamuformula}) we see that
\begin{equation}
  \Delta\mu-\Delta \mu^{(\eps)}
  =\label{deltamu_diff_formula}
  - 2\nabla\fii_\eps\cdot\nabla\mu  - (\mu-1)\Delta\fii_\eps + (1-\fii_\eps)\Delta\mu.
\end{equation}
Using (\ref{fiider1}), (\ref{fiiepspnorm}) and assumption (\ref{MUEST2}) we get
\begin{eqnarray*}
  \left\| \frac{\partial\fii_\eps}{\partial x}\,\,\frac{\partial \mu}{\partial x}\right\|_p
  &=& 
  2\eps^2\left(\int_{\R^2} \left|\frac{\partial\mu(z)}{\partial x}\right|^p |x|^p\fii_\eps(z)^p dz \right)^{1/p}\\
  &\leq& 
  2\eps^2\left(\int_{\R^2} \langle z\rangle^{-2p} |z|^p \fii_\eps(z)^p dz \right)^{1/p}\\
  &\leq& 
  2C\eps^2\|\fii_\eps\|_p \leq
  2C\eps^{2-2/p} \|\fii\|_p,
\end{eqnarray*}
and the same estimate holds for $x$ replaced by $y$. Therefore, since $p>1$ implies $2-2/p>0$, we have proved the convergence
\begin{equation}\label{qconverest1}
  \lim_{\eps\ra 0}\|\nabla\fii_\eps\cdot\nabla\mu \|_p = 0.
\end{equation}

Using (\ref{MUEST1}), (\ref{fiider2}) and
(\ref{fiiepspnorm})  with $s=0$ and $s=1$ we get
\begin{eqnarray*}
  \left\| (\mu-1) \, \Delta \fii_\eps \right\|_p
  &=& \nonumber
  4\eps^2 \| (\mu-1) (\eps^2 x^2 +\eps^2 y^2 - 1) \fii_\eps \|_p\\
  &\leq& \nonumber
  4\eps^2 \| (\mu-1) \eps^2|z|^2 \fii_\eps \|_p
  + 4\eps^2\| (\mu-1) \fii_\eps \|_p\\
  &\leq& \nonumber
  C\eps^4 \|  \,|z| \fii_\eps(z)\, \|_p
  + 4\eps^2 \|\mu-1\|_\infty \,\|  \fii_\eps \|_p\\
  &\leq& 
  C\eps^{3-2/p} \| \fii \|_p+C\eps^{2-2/p} \| \fii \|_p,
\end{eqnarray*}
and it follows from $2/p<2$ that 
\begin{equation}\label{qconverest2}
  \lim_{\eps\ra 0}\|(\mu-1) \, \Delta \fii_\eps\|_p = 0.
\end{equation}

We denote the characteristic function of the disc $B(0,\eps^{-1/4})$ by $\chi_{|z|<\eps^{-1/4}}(z)$ and will use the inequality
\begin{equation}\label{chitriangle}
\|f\|_p\leq \|\chi_{|z|<\eps^{-1/4}}f\|_p + \|\chi_{|z|\geq\eps^{-1/4}}f\|_p
\end{equation}
in the sequel. Note that  $1-\fii_\eps(z)=1-\fii_\eps(|z|)$ is monotonically increasing in $|z|$ and that the area of the disc $B(0,\eps^{-1/4})$ is $\pi\eps^{-1/2}$ and that by assumption (\ref{MUESTQ}) we have
\begin{equation}\label{deltamuestimate}
  |\Delta \mu(z)|\leq \|\mu\|_\infty\left|\frac{\Delta\mu}{\mu}\right|
  =\|\mu\|_\infty |q(z)| \leq C\langle z \rangle^{-2}.
\end{equation} 
Estimate using Taylor expansion of the exponential function near $\eps=0$
\begin{eqnarray}
  \|\chi_{|z|<\eps^{-1/4}}\big(1-\fii_\eps\big)\Delta\mu\|_p^p
  &\leq& \nonumber
 \pi\eps^{-1/2} \|(\Delta\mu)^p\|_\infty\,\|\chi_{|z|<\eps^{-1/4}}|1-\fii_\eps|^p\|_\infty\\
  &\leq& \nonumber
  C\eps^{-1/2}\big|1-\exp(-\eps^2\eps^{-1/2})\big|^p\\
  &\leq& \label{apukaava1}
  C\eps^{(3p-1)/2}.
\end{eqnarray}
In the unbounded set $\{z\,:\,|z|\geq\eps^{-1/4}\}$ we use (\ref{deltamuestimate}) and the fact $|1-\fii_\eps(z)|\leq 1$ to compute
\begin{eqnarray}
  \|\chi_{|z|\geq\eps^{-1/4}}\big(1-\fii_\eps\big)\Delta\mu\|_p^p
  &\leq& \nonumber
  C \|\chi_{|z|\geq\eps^{-1/4}}\langle z \rangle^{-2}\|_p^p\\
  &\leq& \nonumber
  C^\prime \int_{\eps^{-1/4}}^\infty r^{-2p}rdr\\
  && \nonumber\\
  &\leq& \label{apukaava2}
  \frac{C^\prime\eps^{(p-1)/2}}{2(p-1)}.
\end{eqnarray}
Now (\ref{apukaava1}) and (\ref{apukaava2}) together with (\ref{chitriangle}) and $p>1$ give
\begin{equation}  \label{qconverest3}
  \lim_{\eps\ra 0}\|(1-\fii_\eps)\Delta\mu\|_p=0,
\end{equation}
and (\ref{qconverest1}), (\ref{qconverest2}), (\ref{qconverest3}) and (\ref{deltamu_diff_formula})  imply (\ref{Lpconvergence3}). 

It remains to prove (\ref{Lpconvergence2}).
The lower bound (\ref{muepspos}) gives for any $z\in \R^2$
\begin{equation}
  \left|\frac{1}{\mu^{(\eps)}} -  \frac{1}{\mu}\right| 
  =
  \left|\frac{\mu-\mu^{(\eps)}}{\mu\mu^{(\eps)}}\right|
  \leq
  c^{-2}|\mu-\mu^{(\eps)}|
  = \label{epsinftyest}
  C\left|\big(1-\fii_\eps\big)\big(\mu-1\big)\right|.
\end{equation}
We compute first inside the disc $B(0,\eps^{-1/4})$ using (\ref{epsinftyest}):
\begin{eqnarray}
&& \nonumber
\left\|\chi_{|z|<\eps^{-1/4}}\Delta\mu\left(\frac{1}{\mu^{(\eps)}} -  \frac{1}{\mu}\right)\right\|_p^p \\
&\leq&\nonumber
C\pi\eps^{-1/2}\left\|(\Delta\mu)^p\right\|_\infty
\left\|\big(\mu-1\big)^p\right\|_\infty 
\|\chi_{|z|<\eps^{-1/4}}|1-\fii_\eps|^p\|_\infty \\
&\leq&\label{apukaavaB2}
  C\eps^{(3p-1)/2},
\end{eqnarray}
where the second inequality follows as in (\ref{apukaava1}). In the complement of the disc $B(0,\eps^{-1/4})$ we use (\ref{mu0_est1}), (\ref{deltamuestimate}), (\ref{epsinftyest}) and $|1-\fii_\eps|\leq 1$ to get
\begin{eqnarray}
\left\|\chi_{|z|\geq\eps^{-1/4}}\Delta\mu\left(\frac{1}{\mu^{(\eps)}} -  \frac{1}{\mu}\right)\right\|_p^p 
&\leq&\nonumber
C\left\|\chi_{|z|\geq\eps^{-1/4}}|\Delta\mu|^p \big(\mu-1\big)^p\right\|_\infty \\
&\leq&\nonumber
C^\prime\left\|\chi_{|z|\geq\eps^{-1/4}}\langle z\rangle^{-3p}\right\|_\infty\\
  &\leq& \nonumber
  C^\prime \int_{\eps^{-1/4}}^\infty r^{-3p}rdr\\
 &&\nonumber\\
&\leq&\label{apukaavaB3}
  \frac{C^\prime\eps^{(3p-2)/4}}{3p-2}.
\end{eqnarray}
Now (\ref{Lpconvergence2}) follows from (\ref{chitriangle}), (\ref{apukaavaB2}) and (\ref{apukaavaB3}).

\end{document}